\documentclass[a4paper,12pt]{amsart} 
\usepackage{amsmath, amssymb, amsfonts, amsthm}
\usepackage{mathrsfs}
\usepackage[utf8]{inputenc}    
\usepackage[T1]{fontenc}       
\usepackage{bbm}
\usepackage[left=2.5cm,right=2.5cm,vmargin=2.5cm]{geometry}
\usepackage{mathrsfs, stmaryrd, mathtools}
\usepackage{hyperref}
\usepackage{wrapfig}
\usepackage{mathtools}
\usepackage{graphicx}
\usepackage{multirow}
\usepackage{amsthm}
\usepackage{xcolor}
\usepackage{textcomp}
\usepackage{manyfoot}
\usepackage{booktabs}
\usepackage{algorithm}
\usepackage{algorithmicx}
\usepackage{algpseudocode}
\usepackage{listings}
\usepackage{soul}
\usepackage{enumerate}
\usepackage{stmaryrd}

\numberwithin{equation}{section}
\theoremstyle{plain}
\newtheorem{theorem}{Theorem}[section]
\newtheorem{proposition}[theorem]{Proposition} 
\newtheorem{lemma}[theorem]{Lemma}
\newtheorem{corollary}[theorem]{Corollary}
\theoremstyle{definition}
\newtheorem{remark}{Remark}
\newtheorem{definition}{Definition}

\DeclarePairedDelimiterX\intff[2]{[}{]}{#1,#2}
\DeclarePairedDelimiterX\intfo[2]{[}{)}{#1,#2}
\DeclarePairedDelimiterX\intof[2]{(}{]}{#1,#2}
\DeclarePairedDelimiterX\intoo[2]{(}{)}{#1,#2}
\DeclarePairedDelimiter{\pars}{(}{)}
\DeclarePairedDelimiter{\bracks}{[}{]}
\DeclarePairedDelimiter{\absolute}{|}{|}
\DeclarePairedDelimiter{\braces}{\langle}{\rangle}

\DeclarePairedDelimiterX{\setof}[2]{\lbrace}{\rbrace}{#1\,{\colon}\,#2}
\DeclarePairedDelimiterX{\bracksof}[2]{[}{]}{#1\,\delimsize|\,#2}
\DeclarePairedDelimiterX{\parsof}[2]{(}{)}{#1\,\delimsize|\,#2}
\DeclarePairedDelimiterXPP\lnorm[2]{}\lVert\rVert{_{#1}}{#2}

\def\n{\mathbb N}
\def\z{\mathbb Z}
\def\r{\mathbb R}
\def\p{{\mathbb P}} 
\def\e{{\mathbb E}}

\def\T{\mathcal T} 
\def\d{\mathrm{dist}}
\def\d{{\tt d}}

\def\Sp{{\tt Sp}} 
\def\xifour{I}

\newcommand{\1}{\mathbbm{1}}

\def\Var{\mathrm{Var}}
\def\Cov{\mathrm{Cov}}

\def\t{{\tt t}}
\def\F{{\mathscr F}} 
\def\a{{\tt a}}
\def\b{{\tt b}}
\def\c{{\tt c}}
\def\t{{\tt s}}
\newcommand\C[1]{C_{#1}}

\title{A central limit theorem for the range of random walk indexed by an infinite tree}

\author{Tianyi Bai}
\address{Tianyi Bai,  AMSS,
Chinese Academy of Sciences,  China}
\email{tianyi.bai73@amss.ac.cn}

\author{Yueyun Hu}
\address{Yueyun Hu, 
LAGA, Universit\'e Paris XIII,  93430 Villetaneuse,
France}
\email{yueyun@math.univ-paris13.fr}

\begin{document}
\begin{abstract}
We are interested in the second-order fluctuations of the size of the range of a random walk indexed by an infinite tree in ${\mathbb Z}^d$.
Consider the Kesten tree with geometric offspring, we view the vertices on the left of its spine (the unique infinite path) as the past, and those on the right as the future. We show that the number of new sites in the future given its past has linear variance when $d>8$, and satisfies a central limit theorem (CLT) with Gaussian limiting distribution when $d>16$. 
The proof relies on the stationarity of the model under depth-first exploration, a general CLT by Dedecker and Merlev\`ede~\cite{dedecker02}, a truncation technique exploiting the local independence of tree structures, and a recursion argument for moment bounds.
\end{abstract}
\maketitle

\section{Introduction}

Let $(S_n)$ be a centered random walk on $\z^d$, denote its range by
$
{\mathcal S}_n=\{S_1,\dots,S_n\}.
$
The study of fluctuations of $\#{\mathcal S}_n$, that is,
\begin{align}\label{eq:srw}
\frac{\#{\mathcal S}_n-\mathbb E[\#{\mathcal S}_n]}{\sqrt{\Var(\#{\mathcal S}_n)}},
\end{align}
has a long history in probability theory. 
Jain and Pruitt~\cite{jain1971} first showed that \eqref{eq:srw} has a Gaussian limit when the random walk has finite variance and $d\ge 3$. Le~Gall and Rosen~\cite{LeGallRosen91} extended the existence of limiting distribution of \eqref{eq:srw} to $\alpha$-stable random walks in all dimension $d$, where the limiting distribution is the renormalized self-intersection local time  of an $\alpha$-stable process  if $d< \frac{3\alpha}{2}$, and is {Gaussian} if $d \ge \frac{3\alpha}{2}$. 
More recently, attention has turned to the capacity of the range, 
which provides a geometric description of the range and its intersection probability with another independent random walk.  Asselah, Schapira, and Sousi~\cite{AsselahSchapiraSousi2018} and Schapira~\cite{Schapira2020}
proved a CLT for capacity of the range of the simple random walk in dimensions $d \ge 5$ with a {Gaussian} limiting distribution. 
In dimension $d=4$, Asselah, Schapira, and Sousi~\cite{AsselahSchapiraSousi2019} obtained a CLT with a {non-Gaussian} limiting distribution. See also Cygan, Sandrić, and Šebek~\cite{CyganSandricSebek2019} for a CLT on the capacity of the range of an  $\alpha$-stable random walk  when $d> \frac{5\alpha}{2}$.

In the above-mentioned works, we note that the independence of increments plays an important role.
Here, we are interested in the range of branching random walks (BRW), or equivalently, random walks indexed by trees. In contrast to classical random walks, the increments of a BRW, when the vertices are listed in any  sequential order, 
are no longer independent to each other. Our primary motivation is to establish a CLT for such a family of models without independent increments. 

Specifically, let $\theta$ be a centered probability distribution on $\mathbb{Z}^d$. Given any discrete planar tree $\mathcal{T}$ rooted at $\varnothing$, we define a $\mathbb{Z}^d$-valued random process $V_{\mathcal{T}} = (V_{\mathcal{T}}(u))_{u \in \mathcal{T}}$ as follows. Assign to each edge $e$ of $\mathcal{T}$ an independent random variable $X(e)$, all distributed according to $\theta$. Set $V_{\mathcal{T}}(\varnothing) := 0$, and for any $u \ne \varnothing$, define $V_{\mathcal{T}}(u)$ as the sum of $X(e)$ over the edges $e$ along the unique simple path in $\mathcal{T}$ from $\varnothing$ to $u$. We refer to $V_{\mathcal{T}}$ as a BRW  indexed by $\mathcal{T}$.

Let $\mu=(\mu(k))_{k\ge 0}$ be a probability distribution on $\z_+$ with $\mu(1)\neq1$ and $\sum_{k=0}^\infty   k\mu(k)=1. $ Let $\T^c$ be a critical Galton-Watson tree with offspring $\mu$; that is, starting at the root $\varnothing$, each particle branches into $k\ge 0$ particles with probability $\mu(k)$, independently. It is well known that $\T^c$ is almost surely finite.
Let $\T^{(c,n)}$ be distributed as $\parsof{\T^c}{\#\T^c=n}$, where we consider only those $n$ such that $\mathbb P(\#\T^c=n)>0$. Denote its range by
\[
{\mathcal R}^{(c)}_n:=\setof{V_{\T^{(c,n)}}(u)}{u\in\T^{(c,n)}}.
\]
Le Gall and Lin \cite{LeGall-Lin-lowdim, LeGall-Lin-range} and Zhu \cite{Zhu-cbrw} established the precise asymptotic behavior of $\#{\mathcal R}^{(c)}_n$ under mild regularity assumptions on $\theta,\mu$,
\begin{equation}\label{LG-Lin}
\begin{cases}
  \frac{1}{n} \#{\mathcal R}^{(c)}_n \xrightarrow{\mathrm{(p)}} {c_{\theta, \mu, d}}, & \text{if } d \ge 5, \\
  \frac{\log n}{n} \#{\mathcal R}^{(c)}_n \xrightarrow{L^2} {c_{\theta, \mu, d}}, & \text{if } d = 4, \\
  n^{-d/4} \#{\mathcal R}^{(c)}_n \xrightarrow{(d)} {c_{\theta, \mu, d}}\lambda_d(\mathfrak{R}), & \text{if } d \le 3,
\end{cases}
\end{equation}
where $c_{\theta, \mu, d} > 0$ is a constant, $\lambda_d$ denotes the Lebesgue measure on $\mathbb{R}^d$, and $\mathfrak{R}$ denotes the support of the rescaled integrated super-Brownian excursion (ISE). In particular, the dimension $d=4$ plays the role of a critical dimension for the growth of $\#{\mathcal R}^{(c)}_n$. 

It is a natural question to investigate the behavior of $\#{\mathcal R}^{(c)}_n-\mathbb E[\#{\mathcal R}^{(c)}_n]$ when $d \ge 4$, and in particular to seek a  CLT  at least in high dimensions. 
In this work, we establish a CLT for the range of a  BRW  $V_{\T_\infty}$ indexed by the Kesten tree $\T^\infty$, the critical {Galton-Watson} tree \( \T^c \) conditioned to be infinite. The Kesten tree $\T^\infty$ arises as the local limit of \( \T^c \) conditioned on having large height or size.  There is a close relationship between the range of \( V_{\T^\infty} \) and ${\mathcal R}^{(c)}_n$, see  Le Gall \cite{LeGall10} and Zhu \cite{Zhu-cbrw} for more specific explanation and application.

The CLT appears to be an open problem, and remains difficult even in the simplest setting, where $\mu$ is geometric and $\theta$ is a step function of a simple random walk, i.e.
\begin{equation}\label{eq:assmption}
\begin{cases}   
 \mu(k)=2^{-k-1}, & \text{ for all } k\ge 0; \\
\theta(v)=\frac 1{2d}, &\text{ for all } v\in\z^d, \, |v|=1.
 \end{cases}
\end{equation}
In present paper, we adopt this choice for $(\mu,\theta)$ and construct the infinite model $(\T^\infty,V_{\T^\infty})$ as follows.  \begin{enumerate}
\item
Let $\T_0^{\pm},\T_1^{\pm},\dots$ be i.i.d. with the same law as $\T^c$. 
\item Let $\Sp:=\{\varnothing_0,\varnothing_1, \varnothing_2, \cdots\}$, with $\varnothing_0:=\varnothing$,  be an infinite half-line called the spine. For each $i\ge 0$, attach $\T_i^{-}$ and $\T_i^{+}$ to $\varnothing_i$ on the left and right sides of $\Sp$, respectively.
\item Denote $\T^{+}:=\T_0^{+}\cup\T_1^{+}\cup\dots$, $\T^{-}:=\T_0^{-}\cup\T_1^{-}\cup\dots$ and $\T^\infty=\T^+\cup\T^-$. {In particular, we remark that $\T^+\cap\T^-=\Sp$.} Let $V_{\T^\infty}$ be the corresponding branching random walk.
\item Explore the tree $\T^\infty$ in the depth-first order by
\[
...., u_{-2}, u_{-1}, u_0,u_1, u_2 \dots,
\]
as illustrated in Figure \ref{fig1}, and denote
\[
V(k):=V_{\T^\infty}(u_k)\in\mathbb Z^d, \qquad k \in \z.
\]
  \end{enumerate}

The BRW $(V(k))_{k\in \z}$ is the two-sided version of the discrete snake studied by Le Gall and Lin \cite{LeGall-Lin-range}. 
As is shown in \cite{LeGall-Lin-range, zhu2017critical, bai-wan}, $V$ is invariant under translations in the sense that \begin{equation}\label{eq:invariance}
(V(i+j)-V(i))_{j\in\z}\overset{d}{=}(V(j))_{j\in\z}, \qquad \forall \, i\in \z.
\end{equation}

The natural counterpart for $\#{\mathcal R}^{(c)}_n$ on $\T^\infty$ is 
\begin{equation}   
R_n:= \#\{V(1),\dots,V(n)\}=\sum_{i=1}^n \1_{\{V(i)\not\in\{V[1,i)\}}, \label{def-Rn}
 \end{equation}
where $V[i,j):=\setof{V(k)}{k\in [i,j)\cap \z}$ with similar notations for $V[i, j]$, $V(i, j]$, and $V(i,j)$. Indeed, $R_n$ coincides with the size of the range of the discrete snake studied in \cite{LeGall-Lin-range}. 
In this paper, we study a modified version of it,
\begin{equation}    
Y_n:=\sum_{i=1}^n \1_{\{V(i)\not\in V(-\infty,i)\}}=\#(V[1,n]\setminus V(-\infty,0]), \label{def-Yn}
\end{equation}
where, by \eqref{eq:invariance},  the summands $(\1_{V(i)\not\in V(-\infty,i)})_{i\in\z}$ are identically distributed, in other words, forming a stationary sequence. A  slightly modified version of the  ergodic theorem in Le Gall and Lin \cite[Theorem 4]{LeGall-Lin-range} states that for $d\ge 5$, as $n\to\infty$, \begin{equation}    \frac{Y_n}{n} \, \xrightarrow{L^1}\, c_d, \label{EYn} \end{equation}
where $c_d:= \p(0 \not\in V(-\infty, 0))$, and  $c_d$ is positive due to Benjamini and Curien \cite{benjamini-curien}. 

We aim to prove a CLT for $Y_n$. Denote 

{\[
\F_0:=\sigma\{\T^-,V(-\infty,0]
\},
\]} 
The main result is stated below.  

\begin{theorem}\label{T:CLT-Yn} 
Assume \eqref{eq:assmption}. When $d>8$, there is a constant $\kappa=\kappa_d>0$ such that
\begin{equation}    
\lim_{n\to\infty}\frac{\Var(Y_n)}{n}=\kappa. 
\label{kappaexistspositive}
\end{equation}
When $d>16$, for every continuous real function $\varphi$  satisfying $\sup_{x\in\r}|\varphi(x)/(1+x^2)|<\infty$, we have
\begin{align}\label{eq:main2}
\e \left|\e \Big[ \varphi\left(\frac{ Y_n- \e[Y_n]}{\sqrt{n}}\right) \,\big|\, {\F_0}\Big] - \e[\varphi({\tt g}_\kappa)] \right| \to 0, \quad n\rightarrow\infty,
\end{align}
where  ${\tt g}_\kappa\sim {\mathcal N}(0, \kappa)$ denotes a centered Gaussian random variable with variance $\kappa$. 
In particular, $\frac{ Y_n- \e[Y_n]}{\sqrt{n}}$ converges in law to $\mathcal N(0,\kappa)$. 
\end{theorem}

The proof of \eqref{eq:main2} relies on an application of CLT for sums of stationary sequences, adapted from Dedecker and Merlev\`ede~\cite{dedecker02} as follows. 
For every real-valued random variable $X$ with finite mean, denote 
\[
\braces{X}:=X-\e[X].
\]
By \cite[Theorem 1]{dedecker02}, \eqref{eq:main2}
follows once we show the existence of a positive constant $\kappa$ such that \begin{align} 
&\left\| \e\Big[\frac{\braces{Y_n}}{\sqrt{n}} \big| \F_0\Big]  \right\|_1 \to  0, \qquad n \to\infty; \label{cond-a}
\\
& \left\| \e\bracksof*{\frac{\braces{Y_n}^2}{n} -\kappa}{\F_0}\right\|_1 \to 0, \qquad n \to\infty, \label{cond-b}\\
&\left(\frac{\braces{Y_n}^2}{n}\right)_{n\ge 1} \mbox{ is uniformly integrable.} \label{cond-c} 
\end{align}

In the proofs of  \eqref{cond-a}, \eqref{cond-b} and \eqref{cond-c}, we make heavy use of the translational invariance property  in \eqref{eq:invariance}. The main difficulty arises from the lack of the independence and Markov property in $(Y_n)$. Our main idea is to apply  a truncation technique to the summand in $Y_n$, where we replace $\1_{\{V(i)\not\in V(-\infty,i)\}}$ by a local quantity $\xi_i^k$, see \eqref{xi-ki}. This approximation is sufficiently accurate, as is shown in Lemma \ref{lem:xi_to_infty}, while the locality property ensures that $\xi_i^k$ and $\xi_{i'}^{k'}$ become independent whenever the graph distance between $u_i$ and $u_{i'}$ exceeds $k+k'$, see Lemma \ref{lem:independent_xi}.

\begin{remark}
\label{rmk1.2}
{
(i) The dimension restriction $d>8$ in  \eqref{kappaexistspositive} is a technical requirement used in the proof of Lemma \ref{l:cov2}. As in most BRW studies, the difficulty in comparison to simple random walk is the dependence induced by the tree structure. If one is able to show that this additional dependence is negligible, then by the argument of \cite[Eq. (4)]{jain1971}, $\Var(Y_n)$ should be of the scale
\begin{align*}   
\sum_{x\in \z^d} \sum_{i=0}^n\sum_{m=1}^i G(x)^2\mathbb P(V_m=x)
&\asymp \sum_{i=0}^n\sum_{m=1}^i\mathbb E[|S_{\lfloor\sqrt m\rfloor}|^{8-2d}]
\\
& \asymp \sum_{i=0}^n\sum_{m=1}^im^{\frac{4-d}{2}}\asymp\begin{cases}
n\log n, &d=6\\
n, &d\ge 7
\end{cases},
 \end{align*}
where $G(x)\asymp |x|^{4-d}$ is the Green's function of the BRW $(V_n)$.
From this we conjecture that $\Var(Y_n)\asymp n$ when $d\ge 7$. 
}

(ii) 
{We show in Lemma \ref{lem:condition_1} that \eqref{cond-a} is satisfied whenever $d>10$. Establishing \eqref{cond-b} and \eqref{cond-c} is more technical, where we need $d>16$; see Lemma \ref{Lemma:cond-c} and Proposition \ref{P:4thmoment}. }

(iii)
{More precisely, in Lemma \ref{Lemma:cond-c},  $d>16$ is only technically required for \eqref{maxXi-j}. However, for 
Proposition \ref{P:4thmoment} this restriction is more intrinsic.
The estimate of the fourth moment requires that four independent pieces of  BRWs stay ``untangled'' from each other;  each BRW has Hausdorff dimension $4$, thus naturally one needs $d>4\times 4$. 
To improve this condition, one has to prove \eqref{cond-c} without using the fourth moment. }

(iv) The exploration sequence $(u_k)$ can visit the same vertex several times, and that it takes $2n$ times to complete a subtree $\T^{a}_i$ if $\#\T^a_i=n$, $a\in \{+, -\}$. 
In other words, we use the contour process (see Section \ref{sec:2.1}) 
to encode the infinite tree $\T^\infty$, this choice provides several simplifications,  such as the locality in the graph distance, see e.g.   Lemma \ref{lem:xi_to_infty}.  
 
(v) The displacement distribution $\theta$ can be replaced by more general step distributions, after slight modification of Section \ref{subsection:kappa>0}. We choose this simplest case for the sake of brevity. However, the geometric distribution $\mu$ plays a more essential role in our arguments: it is only when $\mu$ is geometric that $\T^+$ and $\T^-$ are independent. In the key Lemma \ref{lem:independent_xi}, we use this property for the independence of \eqref{eq:sigma1} and \eqref{eq:sigma11}.

(vi) The stationarity in the summands of $Y_n$ may be important beyond purely technical considerations. Indeed, we are inclined to believe that the CLT for $R_n$ (in \eqref{def-Rn}), if it holds at all, does not yield a {Gaussian} limit even in the sufficiently high dimensional cases.  More precisely,  it is not hard to show that $\upsilon:=\lim_{n\rightarrow\infty}\frac 1{\sqrt n}\e[R_n-Y_n]\in (0,\infty)$ exists, and we conjecture that 
\begin{align*}
\kappa':=\lim_{n\rightarrow\infty}\frac 1 n\Var(R_n)>0\text{ exists},  \text{ and }\kappa'\ne \kappa.
\end{align*}
Admitting this conjecture, if $R_n$ has CLT with a {Gaussian} limiting distribution,
by the fact that $\frac{R_n-\e[R_n]}{\sqrt n}\ge \frac{Y_n-\e[Y_n]}{\sqrt n}-\frac{\e[R_n]-\e[Y_n]}{\sqrt n}$,
we would have a random variable distributed as $\mathcal N(0,\kappa')$ stochastically dominating another one distributed as $\mathcal N(-\upsilon,\kappa)$, which is impossible.

(vii) We expect that the methods developed in the present work, including the truncation technique, can also be adapted to study other natural quantities related to the range of the BRW, such as its capacity.

 \end{remark}

The rest of the paper is organized as follows. In Section \ref{s:pre}, we collect some estimates on $\T^\infty$ and $V_{\T^\infty}$. Section \ref{s:var} is devoted to the study of $\mbox{Var}(Y_n)$, where we prove the existence of $\kappa=\lim_{n\to\infty} \frac1{n} \mbox{Var}(Y_n)$ in Section \ref{subsection:kappaexists}, and show that $\kappa>0$ is not trivial in 
Section \ref{subsection:kappa>0}.  The conditions \eqref{cond-a} and \eqref{cond-b} are proved in Section \ref{s:conda+b}, and finally,  we prove \eqref{cond-c} in Section \ref{s:cond-c}. 

For notational brevity,  we write 
$V({\tt t})= \{V(u): u \in {\tt t}\}$ for any ${\tt t} \subset \T^\infty$. For instance, we have $V(\T^-)=V(-\infty, 0]$. 
 We write $f(k) \lesssim_\beta g(k)$ if for some fixed parameter $\beta$, there exists a constant $c=c_\beta>0$ such that $f(k) \le c \,g(k)$ for all $k$, and $f(k)\lesssim g(k)$ if $c$ is a numeric constant not depending on any parameter except for the dimension $d$. 

\section{Preliminaries}\label{s:pre}

\subsection{The contour process}\label{sec:2.1}
For any vertex $u\in \T^\infty$, we define its parent as 
the adjacent vertex on the unique simple path from $u$ to infinity, and when $v$ is the parent of $u$, we say $u$ is a child of $v$. For instance,  $\varnothing_{k+1}$ is the parent of $\varnothing_k$.    
Given $\T^\infty$, we define its contour process $C$ by taking $\C{0}=0$ and for every $n\in\z$,
\begin{equation}\label{eq:C}
\C{n+1}-\C{n}=\begin{cases}
1,&\text{if $u_{n+1}$ is a child of $u_n$}\\
-1,&\text{if $u_{n+1}$ is the parent of $u_n$}
\end{cases},
\end{equation}
as is illustrated in Figure \ref{fig1}.

\begin{figure}[ht]
\includegraphics[height=4.8cm]{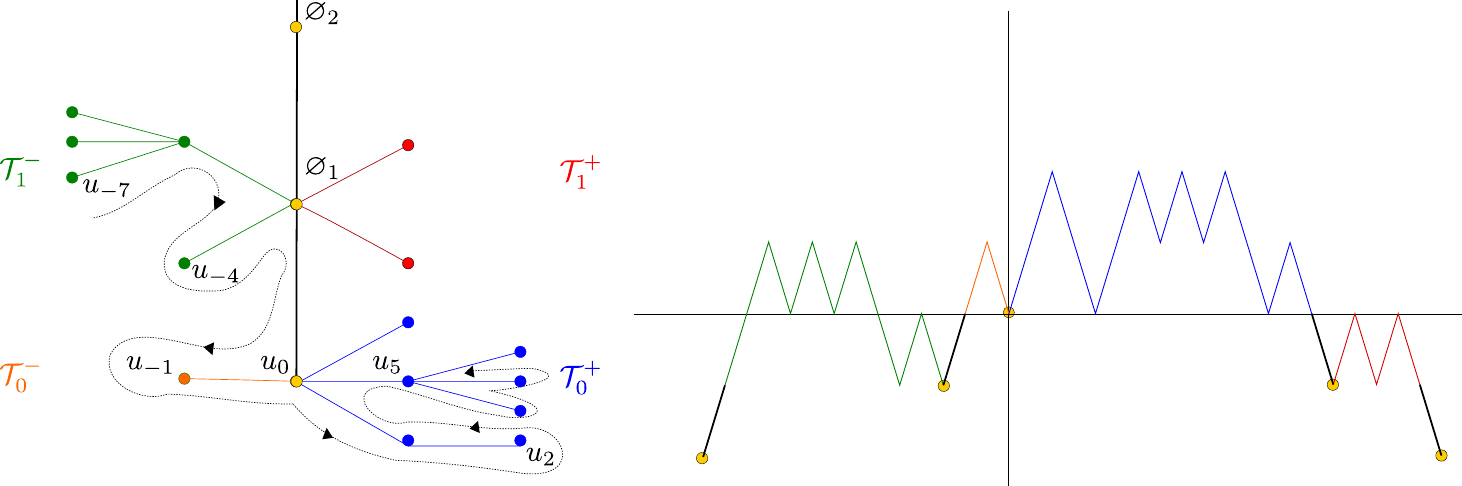}
\caption{\small The infinite tree $\T^\infty$ and its contour process. In the picture,   $u_{-2}=u_0=\varnothing$, $u_{-3}=u_{-5}=\varnothing_1$.}
\label{fig1}
\end{figure}

\begin{remark}\label{rmk:determine}
It is easy to deduce from \eqref{eq:C} and Figure \ref{fig1} that 
\begin{enumerate}
\item $(\C{n})_{n\in\z}$ determines $\T^\infty$;
\item the graph distance between $u_i,u_j$ ($i\le j$) is 
\begin{align}\label{eq:determine}
\d(u_i,u_j)=\C{i}+\C{j}-2\min_{i\le k\le j}\C{k}.
\end{align}
\end{enumerate}
See  Marckert and Mokkadem \cite{marckert-mokkadem} and Le Gall \cite[Section 1.1]{LeGall05} 
for this construction on $\z_+$, but there is no difficulty in extending to $\z$.
\end{remark}

Given the assumption \eqref{eq:assmption} that offspring distribution $\mu$ is geometric, the following lemma is immediate:
\begin{lemma}\label{lem:C}
The contour process $(\C{n})_{n\in\z}$ is a two-sided simple random walk, that is, $(\C{n})_{n\ge 0}$ and $(\C{-n})_{n\ge0}$ are two independent simple random walks on $\z$.
\end{lemma}

Moreover,
we have the following fundamental estimates:
\begin{lemma} \label{lem:estimate-contour} For every $\gamma>2$, 
every $n, k\ge 1$,
\begin{align}     \p(\C{n}-2\min_{0\le i\le n}\C{i}\le  k) & \lesssim \, k^3 \, n^{-\frac 3 2}, \label{proba:C-2minC} \\
 \e[ (1+\C{n}-2\min_{0\le i\le n}\C{i})^{-\gamma}] & \lesssim_\gamma \, \begin{cases}
 n^{-\gamma/2}, & \mbox{if $\gamma< 3$}, \\
 n^{-3/2} \log (1+n),  & \mbox{if $\gamma= 3$}, \\
 n^{-3/2}, & \mbox{if $\gamma>3$}. \end{cases}
 \label{esp:C-2minC}
\end{align}
\end{lemma}
\begin{proof}  Let ${\mathcal M}:= \C{n}-2\min_{0\le i\le n}\C{i}$. 
By \cite[equation (32)]{LeGall-Lin-range},
for any $m\ge 0$, $$ \p({\mathcal M}=m)= \frac{2(m+1)^2}{n+m+2} \p(C_n=m).$$

For \eqref{proba:C-2minC}, it is enough to sum over $m$ and use the well-known fact that $$ \max_{m\in \z} \p(C_n =m) \lesssim n^{-1/2}.$$

Moreover,
\begin{align*} \e[ (1+ {\mathcal M})^{-\gamma}] 
&= \sum_{m=0}^\infty \frac{2(m+1)^{2-\gamma}}{n+m+2} \p(C_n=m)
\\
& 
\lesssim  n^{-3/2} \sum_{m=0}^{\lfloor n^{1/2}\rfloor} (m+1)^{2-\gamma}+ \frac1{n} \e \bracks*{ (1+C_n)^{2- \gamma} \1_{\{C_n \ge \lfloor n^{1/2}\rfloor\}}}
\\
&\lesssim
n^{-3/2} \sum_{m=0}^{\lfloor n^{1/2}\rfloor} (m+1)^{2-\gamma}+ n^{-\gamma/2}.   \end{align*}
This yields \eqref{esp:C-2minC}.    \end{proof}

\subsection{Estimates on \texorpdfstring{$\T^\infty$}{}} In this section, we introduce some notation and present  useful estimates on $(\T^\infty,V_{\T^\infty})$.  
Let $\d$ denote the graph distance on $\T^\infty$. 
For $j, \ell\ge 0$, denote by $Z_\ell^{(j)}$ the population of the $\ell$-th generation of $\T^-_j$. We write $Z_\ell=Z^{(0)}_\ell$ for brevity, then $(Z_\ell)_{\ell\ge0}$ is a  Galton-Watson process with  $\e[Z_\ell]=1$. 

\begin{lemma}\label{lem:moment_geo}
For every fixed $p\ge1$ and for all $n\ge 1$,
\begin{align}    
\mathbb E\bracks*{\pars*{\sum_{i\le 0}\1_{\{\d(u_i,u_0)\le n\}}}^p} 
&\lesssim_p \,  n^{2p},  \label{eq:moment_geo1}
\\
\mathbb E\bracks*{\pars*{\sum_{i\le 0}\1_{\{\d(u_i,u_0)= n\}}}^p} 
&\lesssim_p \,  n^{p}. \label{eq:moment_geo2}
\end{align}
\end{lemma}

Since $\T^-$ and $\T^+$ have the same distribution, the corresponding results hold if we replace  $\sum_{i\le  0}$ by $\sum_{i\ge  0}$ in \eqref{eq:moment_geo1} and \eqref{eq:moment_geo2}.

\begin{proof} We will use the following form of Rosenthal's inequality (\cite{Rosenthal}): Let $p>1, n\ge 1$,  and $\eta_0, ..., \eta_n$ be independent real-valued random variables with finite $p$-moments. Then \begin{equation}   
\e \Big[\Big|\sum_{i=0}^n \eta_i \Big|^p\Big] \le 2^{p^2} \max \Big( \sum_{i=0}^n \e [|\eta_i|^p], (\sum_{i=0}^n \e [|\eta_i|])^p\Big). \label{Rosenthal} \end{equation}

By Remark \ref{rmk1.2}(iv),
\[
 \sum_{i\le 0}\1_{\{\d(u_i,u_0)\le n\}}  \le 2\sum_{j=0}^n\sum_{\ell=0}^{{n+1}} Z_\ell^{(j)}.
\]

\noindent Applying \eqref{Rosenthal} to $\eta_j:= \sum_{\ell=0}^n Z_\ell^{(j)}$, for $0\le j \le n$, and noticing that $\e[\eta_j]= n+1$, we obtain  that $$
\e\Big[ \Big(\sum_{i\le 0}\1_{\{\d(u_i,u_0)\le n\}} \Big)^p\Big] \, \lesssim_p\, n \e \big[\eta_0^p\big] + n^{2 p}. $$

The proof of \eqref{eq:moment_geo1} reduces to show that for any $n,p\ge 1$, 
\begin{equation} \e \big[\eta_0^p\big] \lesssim_p n^{ 2p -1}.  \label{eta-pmoment} 
\end{equation}

Since our tree has explicit geometric offspring, we know (see,  for instance,  \cite[Section 1.4]{athreya2012branching}) that the {generating} function of $Z_\ell$ is 
\[
f_\ell(s):=\mathbb E[s^{Z_\ell}]=1-\frac{1-s}{(1-s)\ell+1},
\]
and by taking the first $p$-th derivatives we obtain,    for all $n\ge 1$,
\begin{align}\label{eq:moment02}
\mathbb E[Z_n^p]\lesssim_p n^{p-1}.
\end{align}

Notice that $(Z_\ell)_{\ell \ge 0}$ is a martingale. The Doob's $L^p$-inequality yields that $$ \e \Big[ \max_{0\le i \le n} Z_i^p \Big] \lesssim_p \e[Z_n^p]\lesssim_p  n^{p-1}.$$
It follows that $\e \big[\eta_0^p\big] \le (n+1)^p \e  \big[ \max_{0\le i \le n} Z_i^p \big] \lesssim n^{2p-1}$. This proves \eqref{eta-pmoment}, and completes the proof of  \eqref{eq:moment_geo1}.

For \eqref{eq:moment_geo2}, we note that for any $n\ge 1$, 
\begin{equation}  \sum_{i\le 0}\1_{\{\d(u_i,u_0)= n\}}
\le   \sum_{j=0}^n (Z^{(j)}_{n-j} + Z^{(j)}_{n-j+1}), \label{eq:d=n}
\end{equation} where the sum $Z^{(j)}_{n-j} + Z^{(j)}_{n-j+1}$ arises from the fact that, at each visit to a child of a vertex, the contour process returns. The proof of \eqref{eq:moment_geo2} reduces to show that 
$$ \e\Big[\Big( \sum_{j=0}^n Z^{(j)}_{n-j}\Big)^p\Big] \lesssim_p n^p. $$ 

Applying \eqref{Rosenthal} to $\eta_j= Z^{(j)}_{n-j}$ for $0\le j \le n$, we have $$ \e\Big[\Big( \sum_{j=0}^n Z^{(j)}_{n-j}\Big)^p\Big] \lesssim_p  \sum_{j=0}^n \e[(Z_{n-j})^p] + \Big(\sum_{j=0}^n \e[Z_{n-j}]\Big)^p  \lesssim_p n^p,$$
by using again \eqref{eq:moment02} and the fact that $\e[Z_{n-j}]=1$. This completes the proof of \eqref{eq:moment_geo2}.
\end{proof}

\begin{lemma} Let $d\ge 3$. We have for all $j\ge 1$, \begin{equation}  \p\Big(0\in V(\T^-_j)\Big) \lesssim j^{(2-d)/2}. \label{eq:0inVj}  \end{equation}   \end{lemma}

\begin{proof}  

Let $S_n:=V(\varnothing_n)$, then $(S_n)_{n\ge 0}$ is a simple random walk on $\z^d$.
By considering the possible position of $S_j$, we get that   
$$\p(0\in V(\T^-_j)) 
 \le \sum_{x\in \z^d}\p(S_j=x) \p(-x \in V(\T_0^-)).$$
 
Let $Z_\ell$ denote the number of individuals in the $\ell$-th generation of $\T_0^-$. We have $$\p(-x \in V(\T_0^-))
 \le
  \sum_{\ell=0}^\infty \sum_{m=1}^\infty  m \p(Z_\ell=m) \p(S_\ell=-x)
  = \sum_{\ell=0}^\infty  \p(S_\ell=-x).$$
It follows that 
\begin{align*} \p(0\in V(\T^-_j)) 
&\le   \sum_{x\in \z^d}\p(S_j=x)  \sum_{\ell=0}^\infty   \p(S_\ell=-x)
=\sum_{\ell=0}^\infty\p(S_{j+\ell}=0)
\lesssim j^{(2-d)/2},
\end{align*}
proving the lemma. \end{proof}

We end this section by introducing a truncation argument which will be useful throughout the paper. 
Recall that $\d$ denotes the graph distance on $\T^\infty$. For every $i \in \z$,  $k \in \n\cup\{\infty\}$,  define 
\begin{align}    
N_k(i)& :=\{j \in (-\infty, i)\cap \z :  \, \d(u_j,u_i)\le k\},\label{def-neighborhood}
\\
\xi^k_i&:=\1_{\{V(i)\not\in V(N_k(i))\}}, \label{xi-ki}
\end{align}

\noindent with $N_\infty(i)=(-\infty, i) \cap \z$  and $\xi^\infty_i= \1_{\{V(i)\not\in V(-\infty, i)\}}$. Then \begin{equation}    Y_n= \sum_{i=1}^n \1_{\{V(i)\not\in V(-\infty,i)\}}= \sum_{i=1}^n \xi^\infty_i. \label{Yn=xiinfty} \end{equation}

We will repeatedly use the following lemma, which gives an upper bound on the error term when truncating within a given distance $k$:

\begin{lemma}\label{lem:xi_to_infty}
Let $d\ge 5$. We have for all $k\ge 1$ and $i\in\z$, \[
0\le  \e [\xi_i^k- \xi_i^\infty] \lesssim k^{\frac {4-d}2}.
\]
\end{lemma}
\begin{proof}
By \eqref{eq:invariance}, it suffices to take $i=0$, and show that
\begin{align}\label{eq:far}
\mathbb P(V(0)\in V((-\infty,0)\setminus N_k(0)))\lesssim k^{\frac {4-d}2}.
\end{align}

Decompose $\T^-=\T^-_{0}\cup\T^-_{1}\cup\dots$, by union bounds,
\begin{align*}
 \p(0\in V((-\infty,0)\setminus N_k(0)))
&\le \sum_{j= 0}^\infty \mathbb P(0\in V(\T_j^{(-,>k)}))
\\
&\le
\sum_{j= 0}^{k-1} \p(0\in V(\T_j^{(-,>k)})) + \sum_{j=k}^\infty \p(0\in V(\T_j^-)),
\end{align*}

\noindent where     $$ \T_j^{(-,>k)}:= \{u \in   \T_j^- :  \d(u,  \varnothing) > k\}.$$

Note that for every $j< k$, the set $\T_j^{(-,>k)}$ consists of the vertices beyond the $(k-j)$-th generation in $\T^-_{j}$ (if such vertices exist). Recall that  $Z_{k-j}^{(j)}$ denotes the population of the $(k-j)$-th generation in $\T^-_j$. Then $\T_j^{(-,>k)}$ can be viewed as the union of the vertices of $Z_{k-j}^{(j)}$ subtrees,  each distributed (though not independently) as $\T^-_k$.  Note that $\e[Z_{k-j}^{(j)}]=1$, by the union bound, we have that for any $0\le j < k$, 
$$\p(0\in V(\T_j^{(-,>k)})) 
\le
\sum_{m\ge 1}m\mathbb P(Z_{k-j}^{(j)}=m) \p(0\in V(\T^-_k))= \p(0\in V(\T^-_k)).$$

It follows that  \begin{equation*}    
 \mathbb P(0\in V((-\infty,0)\setminus N_k(0))) 
\le
   k\, \mathbb P(0\in V(\T^-_k)) + \sum_{j= k}^\infty\mathbb P(0\in V(\T^-_j)),    
\end{equation*}
which, in view of \eqref{eq:0inVj}, implies  Lemma \ref{lem:xi_to_infty}.\end{proof}

\section{Linear growth of variance: proof of \texorpdfstring{\eqref{kappaexistspositive}}{}}\label{s:var}

This section is devoted to the proof of \eqref{kappaexistspositive}, which we recall here:
\begin{proposition}\label{prop:linear_variance}
Let $d>8$. There exists some $\kappa=\kappa_d>0$ such that 
\begin{equation*}    
\lim_{n\to\infty}\frac{\Var(Y_n)}{n}=\kappa. 
\end{equation*}
\end{proposition}

The proof of Proposition \ref{prop:linear_variance} is divided into two steps. First we show the existence of the limit $\kappa\ge 0$ in Section \ref{subsection:kappaexists}. Then, in Section \ref{subsection:kappa>0},  we adapt some ideas from \cite{AsselahSchapiraSousi2018} to establish that $\kappa>0$.

\subsection{Existence of  \texorpdfstring{$\kappa$}{}}\label{subsection:kappaexists}  
 Recall the definition of $\xi^k_i$ from \eqref{xi-ki}, which can be regarded as a truncated version of $\xi^\infty_i$. The following lemma establishes that one can obtain a certain degree of independence from this truncation.

\begin{lemma}\label{lem:independent_xi}
For any $i\in \z, k_1\ge 0$. The random variable 
$\xi_i^{k_1}$  is independent of
\begin{align}
\label{eq:sigma2}
\sigma\setof{(\xi_j^{k_2}\1_{\{\d(u_i,u_j)=m\}}, \d(u_i,u_j))}{j>i,k_2\ge 0,m\ge k_1+k_2} .
\end{align}
 
\end{lemma}

\begin{proof}
\begin{figure}[ht]
\centering
\includegraphics[height=5cm]{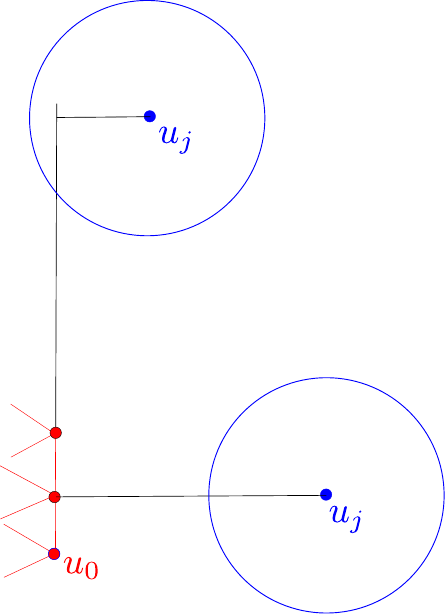}
\caption{\scriptsize As long as $m\ge k_1+k_2$, $\xi_j^{k_2}\1_{\{\d(u_0,u_j)=m\}}$ is always independent of the first $k_1$ subtrees in the past of $u_0$. {In the picture, we take $i=0$, and the two circles illustrate the two possible positions for $u_j$.}}
\label{fig:far}
\end{figure}

By \eqref{eq:invariance}, we can assume $i=0$. 
Recall that $(\varnothing_\ell)_{\ell\ge 0}$ is the spine where on each $\varnothing_\ell$ we attach two independent subtrees $\T_\ell^+,\T_\ell^-$.
The term $\xi^{k_1}_0$ only depends on the initial $k_1{+1}$ subtrees in the past of $u_0=\varnothing$, so we have
\begin{align}\label{eq:sigma1}
\xi^{k_1}_0 \mbox{ is measurable w.r.t. } \sigma\{\T_0^-, \dots, \T_{k_1}^-, V(\cup_{\ell=0}^{k_1} \T_\ell^-)\}.
\end{align}

Let $j>0, k_2\ge 0$ and $m\ge k_1+k_2$.  
It is immediate that 
\begin{align}\label{eq:sigma11}
\{\d(u_0,u_j)=m\}\in\sigma(\T^+).
\end{align}
Moreover, 
when $\d(u_0,u_j)=m\ge k_1+k_2$, as illustrated in Figure \ref{fig:far}, one of the following two cases must hold:
\begin{enumerate}
\item There is some $0\le \ell \le k_1$ such that $\{u_r: r\in N_{k_2}(j)\}\subset \T^+_\ell$  and $\xi_j^{k_2}$ is determined by $V(\T^+_\ell)-V(\varnothing_\ell)$; 
\item $\{u_r: r\in N_{k_2}(j)\}\subset\cup_{\ell>k_1}(\T_\ell^-\cup\T_\ell^+)\cup \{\varnothing_{k_1}\}$, and $\xi_j^{k_2}$ is determined by $V(\cup_{\ell>k_1}(\T_\ell^-\cup\T_\ell^+))-V(\varnothing_{k_1})$.
\end{enumerate}
This shows that $ \xi_j^{k_2} \1_{\{\d(u_0,u_j)= m\}}$ is ${\mathscr H}$-measurable, where $$
{\mathscr H}:=  \sigma(\mathcal T^+,\cup_{\ell>k_1}\T_\ell^-,(V(\T^+_\ell)-V(\varnothing_{\ell}))_{0\le\ell\le k_1},V(\cup_{\ell>k_1}(\T_\ell^-\cup\T_\ell^+))-V(\varnothing_{k_1})).$$

\noindent Since  $\d(u_0, u_j)$ is also $\sigma(\T^+)$-measurable, we have
\begin{equation*}  \sigma\setof{(\xi_j^{k_2}\1_{\{\d(u_i,u_j)=m\}}, \d(u_i,u_j))}{j>i,k_2\ge 0,m\ge k_1+k_2}  \subset {\mathscr H}.  \end{equation*}

By \eqref{eq:sigma1}, $\xi_0^{k_1}$ is independent of ${\mathscr H}$, thus is independent of \eqref{eq:sigma2}.
\end{proof}

\begin{lemma} \label{l:cov2} Let $d>8$.  We have   \begin{equation}    \sum_{n=0}^\infty |\Cov(\xi_0^\infty,\xi_{n}^\infty) | < \infty.  \label{cov-2terms} \end{equation}
\end{lemma}
 
\begin{proof}  
Let  $\alpha_n=\lfloor \d(u_0,u_n)/2\rfloor$. Write  \begin{equation}  \Cov(\xi_0^\infty,\xi_{n}^\infty)
= \Cov(\xi_0^\infty ,\xi_n^{\alpha_n}) +
\Cov(\xi_0^\infty ,\xi_{n}^\infty - \xi_n^{\alpha_n}). \label{eq:alpha1}  \end{equation} 

We first  treat  $\Cov(\xi_0^\infty ,\xi_n^{\alpha_n}) $. Note that \begin{align} 
\Cov(\xi_0^\infty ,\xi_n^{\alpha_n})
=\sum_{k} \Cov(\xi_0^\infty ,\xi_n^{k} \1_{\{\alpha_n=k\}})
=\sum_{k} \Cov(\xi_0^\infty - \xi_0^{k} ,\xi_n^{k} \1_{\{\alpha_n=k\}})
\end{align}
where the last equality holds by Lemma \ref{lem:independent_xi}. 
Note that $\xi_0^\infty - \xi_0^{k}$ is measurable with respect to $\sigma(\T^-,V(\T^-))$,  and $\{\alpha_n=k\}\in\sigma(\T^+)$. Then $\xi_0^\infty - \xi_0^{k}$ and $\{\alpha_n=k\}$ are independent, we deduce by Lemma \ref{lem:xi_to_infty} that for any $k\ge 0$, 
\begin{align*} |\Cov(\xi_0^\infty - \xi_0^{k} ,\xi_n^{k} \1_{\{\alpha_n=k\}})|
&\le 2 \e[|\xi_0^\infty - \xi_0^{k}|   \1_{\{\alpha_n=k\}}]
\\
&= 2 \e[|\xi_0^\infty - \xi_0^{k}| ] \p( \alpha_n=k)\\
&\lesssim (1+k)^{(4-d)/2}\, \p( \alpha_n=k).
\end{align*} 

\noindent Hence $$ | \Cov(\xi_0^\infty ,\xi_n^{\alpha_n})| \lesssim \e[(1+\alpha_n)^{(4-d)/2}]
\lesssim \mathbb E[ (1+\C{n}-2\min_{0\le i\le n}\C{i})^{(4-d)/2}] ,$$ 

\noindent where the last inequality follows from \eqref{eq:determine}. 
This in view of \eqref{esp:C-2minC} yields that for $d>8$, we have \begin{equation}    \sum_n | \Cov(\xi_0^\infty ,\xi_n^{\alpha_n})| < \infty. \end{equation}

For the term $\Cov(\xi_0^\infty ,\xi_{n}^\infty - \xi_{n}^{\alpha_n})$, 
it suffices to check that \begin{equation}    \sum_n \e[|  \xi_{n}^\infty - \xi_n^{\alpha_n} |] =   \sum_n \sum_{k} \e [|  \xi_{n}^\infty - \xi_n^{k} | \1_{\{\alpha_n=k\}}] < \infty.  \label{compa-alpha}\end{equation}

By \eqref{eq:invariance}, $\e [|  \xi_{n}^\infty - \xi_n^{k} | \1_{\{\alpha_n=k\}}]= \e [|  \xi_{0}^\infty - \xi_0^{k} | \1_{\{\d(-n, 0)\in\{2k,2k+1\}\}}],$ and the indicator $\1_{\{\d(-n, 0)\in\{2k,2k+1\}\}}$ is measurable with respect to $\mathcal T^-$.
By the union bound 
\begin{align}  
\e [|  \xi_{0}^\infty - \xi_0^{k} | \, | \T^- ] 
&\le \p\Big( 0 \in V(-\infty, 0) \backslash V(N_k(0)) \, | \T^- \Big) \nonumber
\\
& \le \sum_{i=0}^\infty \sum_{j=(k-i)^+}^\infty Z_j^{(i)} \p( S_{j+i}=0) \nonumber
\\
&\lesssim \sum_{i=0}^\infty \sum_{j=(k-i)^+}^\infty Z_j^{(i)} (j+i)^{-d/2},
\end{align}
where $Z^{(i)}_j$ denotes the number of vertices at $j$-th generation of $\T^-_i$, and we admit $Z^{(i)}_j=0$ when $j<0$.

Since each vertex $u$ is visited $\mathrm{deg}(u)$ times along the sequence $(u_i)$, we have
\[\#\setof{u_i}{i <0,\, \d(u_i, u_0)\in\{2 k,2k+1\}} \le \sum_{i'\ge 0} (Z_{2k-i'+2}^{(i')}+2Z_{2k-i'+1}^{(i')}+Z_{2k-i'}^{(i')}).\]

Therefore   
\begin{align}   
\sum_n \e[|  \xi_{n}^\infty - \xi_n^{\alpha_n} |]   
\lesssim &\sum_{k}\e\bracks*{ \sum_{i=0}^\infty \sum_{j={(k-i)^+}}^\infty Z_j^{(i)} (j+i)^{-\frac d 2} \, \sum_{i'\ge0} (Z_{2k-i'+2}^{(i')}+2Z_{2k-i'+1}^{(i')}+Z_{2k-i'}^{(i')})}\nonumber\\
=&\sum_{k} A_{k,\eqref{eq:gw2}} + \sum_{k} B_{k,\eqref{eq:gw2}},  \label{eq:gw2}
\end{align} 

\noindent with \begin{align*} 
A_{k, \eqref{eq:gw2}}&:=   \sum_{i=0}^\infty \sum_{j={(k-i)^+}}^\infty {\mathbb E}\bracks*{Z_j^{(i)} (j+i)^{-\frac d 2} \, (Z_{2k-i+2}^{(i)}+2Z_{2k-i+1}^{(i)}+Z_{2k-i}^{(i)})},  \\
 B_{k, \eqref{eq:gw2}}&:=  \sum_{i=0}^\infty \sum_{j={(k-i)^+}}^\infty {\mathbb E}\bracks*{Z_j^{(i)} (j+i)^{-\frac d 2} \, \sum_{i'\ne i} (Z_{2k-i'+2}^{(i')}+2Z_{2k-i'+1}^{(i')}+Z_{2k-i'}^{(i')})}. \end{align*}

For the term $A_{k,\eqref{eq:gw2}}$, we have 
\begin{align*}  A_{k,\eqref{eq:gw2}}
&= \sum_{i \le 2k+2} \sum_{j={(k-i)^+}}^\infty  (j+i)^{-d/2} \e[ Z_j^{(i)} (Z_{2k-i+2}^{(i)}+2Z_{2k-i+1}^{(i)}+Z_{2k-i}^{(i)})]
\\
&\lesssim \sum_{i \le 2k+2} \sum_{j={(k-i)^+}}^\infty  (j+i)^{-d/2}  \, j\wedge (2k-i+2) 
\lesssim k^{3- d/2}, \end{align*}

\noindent where we have used the fact that $\e[Z_j^2] =O(j)$. This sum over $k$ converges if $d>8$. 

For the term $B_{k,\eqref{eq:gw2}}$, since $Z^{(i)}$ and $Z^{(i')}$ are independent, we have
\begin{align*}  
B_{k,\eqref{eq:gw2}}
&= \sum_{i\ge 0} \sum_{j={(k-i)^+}}^\infty  (j+i)^{-d/2} \sum_{i'\ne i, i'\le 2k+2} 4
\lesssim {k^{3-d/2}},\end{align*}
which is summable for {$d> 8$}.

Combine \eqref{eq:gw2}   with \eqref{compa-alpha}, we complete the proof. 
\end{proof}

\begin{corollary}\label{coro:var}
Let $d>8$, then
\[
\kappa:=\lim_{n\rightarrow\infty}\frac{\Var(Y_n)}{n}\in[0,\infty)\text{ exists}.
\]
\end{corollary}
\begin{proof}
Denote
\[
a_j:=\frac 1 2 \Var(\xi_j^\infty)+\sum_{i=1}^{j-1}\Cov(\xi_i^\infty,\xi_j^\infty),
\]
then $
\Var(Y_n)=2(a_1+\dots+a_n),
$
and it suffices to show that $\lim_{n\rightarrow\infty}a_n$ exists.

By  \eqref{eq:invariance}, 
\[
a_j=\frac 1 2 \Var(\xi_0^\infty)+\sum_{i=1}^{j-1}\Cov(\xi_0^\infty,\xi_{i}^\infty).
\]
We conclude with \eqref{cov-2terms}.
 \end{proof}

\subsection{Non-triviality of \texorpdfstring{$\kappa$}{}}\label{subsection:kappa>0}
In this subsection, we show that 
$\kappa:=\lim_{n\rightarrow\infty}\frac 1 n\Var(Y_n)$ is positive. 

For each vertex $v\in\T^+$, we denote $\overleftarrow v$ the vertex on the shortest path from $v$ to $\overleftarrow\varnothing$ that is adjacent to $v$. Note that for any $\varnothing_k \in \Sp=\{\varnothing_0,\varnothing_1, \varnothing_2, \cdots\}$ with $k\ge 1$, $\overleftarrow \varnothing_k= \varnothing_{k-1}$. For any vertex $v \not\in \Sp$, $\overleftarrow v$ coincides with the notion of parent defined from the contour process \eqref{eq:C}.

For notational convenience, we artificially add a point $\overleftarrow\varnothing$ attached to the root $u_0=\varnothing$, located at $V(\overleftarrow\varnothing)=\mathbf e_1$ (the unit vector along first coordinate in $\mathbb Z^d$), throughout this section. The choice of $\mathbf e_1$ is  arbitrary. We stress that  $\overleftarrow\varnothing$ is not considered part of $\T^\infty$. The introduction of $\overleftarrow\varnothing$ is meant to ensure the consistency of the forthcoming Definition \ref{defn:bad}.

We consider a slightly modified version of $Y_n$: 
\begin{equation}    
\widetilde Y_n:=\sum_{i=1}^n \1_{\{{V(i)}\not\in V(-\infty,i)\cup \{V(\overleftarrow\varnothing)\}\}}, \label{tilde-Y}
\end{equation}

Notice that $0\le Y_n- \widetilde Y_n\le 1$,   we have \begin{equation}  \lim_{n\rightarrow\infty}\frac 1 n\Var(\widetilde Y_n)=\kappa. \label{var-tildeY}   \end{equation}

The idea for proving the positivity of $\kappa$ is that among $V(1),\cdots, V(n)$, there are certain points -- referred to below as  ``bad points'' --  that make no contribution to $V[0,n]\cup V(\overleftarrow\varnothing)$. 
In Lemma \ref{lem:YN}, we shall use $N_n$ to denote the number of bad points, where we show that it has variance of order $n$, and is independent of the remaining process, which will imply the positivity of $\kappa$.

\begin{definition}\label{defn:bad}
We say a point $v\in\mathcal T^+$ is bad, if
\begin{enumerate}
\item $\d(v,u_0)$ is odd;
\item $\deg(v)=1$, i.e. $v$ is a leaf of $\T^\infty$;
\item $V(v)=V(\overleftarrow{\overleftarrow v})$.
\end{enumerate}
\end{definition}

We remark that every argument of this subsection applies if $\d(v, u_0)$ is chosen to be even instead of odd.

For simplicity, we  write \[
\Delta C(i)=\C{i}-\C{i-1},\quad \Delta V(i)=V(i)-V({i-1}).
\]
Then the following statements are immediate:
\begin{lemma}\label{lem:increment}
\begin{enumerate}
\item
The sequence $(\Delta C(i))_{i\ge 1}$ is i.i.d., taking values in $\{1,-1\}$ with probability $\frac 1 2$ each.
\item
When $\Delta C(i)=1$, $\Delta V(i)$ is a uniform unit vector independent of $\sigma(C(-\infty,i-1],V(-\infty,i-1])$.
\item 
When $\Delta C(i)=-1$, $\Delta V(i)$ is the vector from $u_{i-1}$ to {its parent}, completely determined by $(C(-\infty,i-1],V(-\infty,i-1])$.
\item In particular, $u_{2i-1}$ is bad if and only if $\Delta C(2i-1)=1,\Delta C(2i)=-1$, and {$\Delta V(2i-1)=-\Delta V(2i)$} is the vector from $u_{2i-2}$ to {$\overleftarrow{u_{2i-2}}$}. Therefore,
\[
\mathbb P\parsof{u_{2i-1}\text{ is bad}}{C(-\infty,2i-2],V(-\infty,2i-2]}=\frac 1{8d},
\]
and the event $\{u_{2i-1}\text{ is bad}\}$ is independent of the past $\sigma$-algebra  $\sigma(C(-\infty,2i-2],V(-\infty,2i-2])$.
\end{enumerate}
\end{lemma}

For each $i\ge 1$ such that $u_{2i-1}$ is bad, we delete $u_{2i-1}$ and $u_{2i}$ from the sequence $(u_n)_{n\ge 0}$, and denote the remaining sequence by $(\widehat u_n)_{n\ge 0}$. As is shown in Figure \ref{fig:bad},  $(\widehat u_n)_{n\ge 0}$ is still depth-first exploration sequence of an infinite tree $\widehat\T$ rooted at $u_0$. Define $(\widehat C_{n})_{n\ge 0}$ and $(\widehat V(n))_{n\ge 0}$ to be its contour and spatial positions, and denote by $X_i$ the number of bad points (that have been deleted) between $\widehat u_{2i}$ and $\widehat u_{2i+1}$. In other words, $X_i$ equals half of the amount of points deleted between $\widehat u_{2i},\widehat u_{2i+1}$.

\begin{figure}[ht]
\includegraphics[height=5cm]{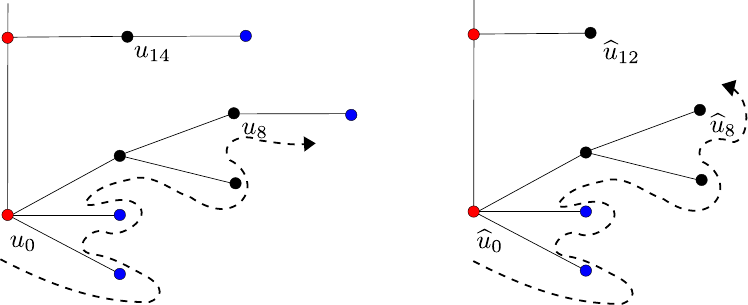}
\caption{\small In the left-side tree, we can enumerate $(\Delta C(2i-1,2i))_i$ as $((1,-1);(1,-1);(1,1);(-1,1);(1,-1);(-1,-1);(-1,1);(1,-1),\dots)$. There are four occurrences of $(1,-1)$, so we have four possible bad points, marked in blue. Suppose that the last two are bad, after deleting them we obtain the right-side tree with $\widehat C=((1,-1);(1,-1);(1,1);(-1,1);(-1,-1);(-1,1)\dots)$. In notation of Lemma \ref{lem:YN}, we have $X_4=X_6=1$, because we delete a bad point after the $4$-th and $6$-th parenthesis in $\widehat C$, or equivalently, after $\widehat u_8$ and $\widehat u_{12}$.}
\label{fig:bad}
\end{figure}

\begin{lemma}\label{lem:YN} \begin{enumerate}
\item
{The random variables $(X_i)_{i\ge 0}$} are i.i.d. geometric with parameter $\frac 1 {8d}$.
\item 
$(X_i)_{i\ge 0}$ is independent of $(\widehat C,\widehat V)$.
\item In particular, 
\[
N_n:=X_0+\dots+{X_{\lfloor \frac n 2 \rfloor}},
\]
\[
\widehat Y_n:=\sum_{i=1}^n\1_{\{\widehat V({i})\not\in \widehat V(0,{i})\cup V(-\infty,0]\cup \{V(\overleftarrow\varnothing)\}\}}
\]
are independent.
\item Recall $\widetilde Y$ from \eqref{tilde-Y}, we have $$\widetilde Y_{n+2N_n}=\widehat Y_{n}.$$
\end{enumerate}
\end{lemma}
\begin{proof}
The  claim (1) follows immediately from Lemma \ref{lem:increment} (4). 
By definition of bad points and $N_n$, $V[0,n+2N_n]\cup V(\overleftarrow\varnothing)=\widehat V[0,n]\cup V(\overleftarrow\varnothing)$, from which we easily obtain claim (4).
The claim (3) is an immediate corollary of claim (2), so it suffices to prove (2).

Let $c_1,\dots,c_{2n}\in \{1,-1\}$, $v_1,\dots,v_{2n}\in \{\text{unit vectors in }\mathbb Z^d\}$, $x_0,\dots,x_n\in\mathbb N$ be arbitrarily fixed, and write $s_n:=x_0+\dots+x_n$ for simplicity.

Let us calculate
\begin{align}\label{eq:long}
p_{\eqref{eq:long}}:= \p\parsof{\Delta\widehat C[1,2n]=c[1,2n],\Delta\widehat V[1,2n]=v[1,2n], X[0,n]=x[0,n]}{\F_0},
\end{align}

\noindent with $$\F_0= \sigma\{C(-\infty,0],V(-\infty,0]\}.$$

We first look at the contour process. The event \eqref{eq:long} implies that $\Delta C$, the increments of the original contour process, is
\begin{equation}\label{eq:contourc}
\begin{aligned}
\Delta C[1,2n+2s_n]
&=((1,-1)^{x_0},c_1,c_2,(1,-1)^{x_1},\dots,c_{2n-1},c_{2n},(1,-1)^{x_{n}})\\
&=:\widetilde c[1,2n+2s_n] ,
\end{aligned}
\end{equation}
where $(1,-1)^x$ means inserting $x$ times of $(1,-1)$. By Lemma \ref{lem:increment}(1),  \[ p_{\eqref{eq:contourc}}:= \p\Big(\Delta C[1,2n+2s_n]= \widetilde c[1,2n+2s_n]\, |\, \F_0\Big) =2^{-2n-2s_n}.
\]

Notice that given $\Delta C[1,2n+2s_n]$, $\Delta V[1,2n+2s_n]$ can be recovered from  $\Delta\widehat V[1,2n]=v[1,2n]$. In other words, $$ 
p_{\eqref{eq:long}}= \p\parsof{\Delta C[1,2n+2s_n]=\widetilde c[1,2n+2s_n],\Delta V[1,2n+2s_n]= \widetilde v[1,2n+2s_n]}{\F_0},
$$

\noindent for some sequence $\widetilde v[1,2n+2s_n]$ which is a deterministic function of $\widetilde c[1,2n+2s_n]$ and $v[1,2n]$.

Now, observe that given  
$\Delta C[1,2n+2s_n]$, by Lemma \ref{lem:increment} (2), $(\Delta V(i)\1_{\{\Delta C(i)=1\}})_{1\le i\le 2n+2s_n}$ equals any fixed (and compatible) sequence with probability
\begin{equation}    
p_{\eqref{deltaVC=1}}:=(2d)^{-\sum_{1\le i\le 2n+2s_n}\1_{\{\Delta C(i)=1\}}}
=(2d)^{-\sum_{1\le i\le 2n}\1_{\{\Delta c(i)=1\}}-s_n}, \label{deltaVC=1}
\end{equation}

\noindent where in the second equality, we have used the fact that on the set $\{\Delta\widehat C[1,2n]=c[1,2n]\}$, $$\sum_{1\le i\le 2n+2s_n}\1_{\{\Delta C(i)=1\}}=\sum_{1\le i\le 2n}\1_{\{\Delta \widehat C(i)=1\}}+ s_n= \sum_{1\le i\le 2n}\1_{\{\Delta c(i)=1\}}+s_n.$$

 By Lemma \ref{lem:increment}(3), the remaining increments $(\Delta V(i)\1_{\{\Delta C(i)=-1\}})_{1\le i\le 2n+2s_n}$ are deterministic function of $c[1, 2n], v[1, 2n]$ and $V(\varnothing_k)_{0\le k \le n}$. Hence there is some deterministic measurable function $\Psi$ such that 
\begin{align}    p_{\eqref{eq:long}}
&= p_{\eqref{eq:contourc}} p_{\eqref{deltaVC=1}} \Psi(c[1, 2n], v[1, 2n], V(\varnothing_k)_{0\le k \le n})\nonumber
\\
&=(8d)^{-s_n}  \, 2^{-2n}(2d)^{-\sum_{1\le i\le 2n}\1_{\{\Delta c(i)=1\}}} \Psi(c[1, 2n], v[1, 2n], V(\varnothing_k)_{0\le k \le n}).\label{eq:315}
\end{align}

Since $\p(X[0,n]=x[0,n])= {(1-\frac 1 {8d})^{n+1}}(8d)^{-s_n}$, and the remaining terms of \eqref{eq:315} does not depend on $x[0,n]$, we take the expectation of $p_{\eqref{eq:long}}$ and obtain the independence of $X[0, n]$ and $(\Delta \widehat C[1, 2n], \Delta \widehat V[1, 2n])$.   Claim (2) then follows. 

\end{proof}

We are now ready to prove that the variance is indeed linear:

\begin{proof}[Proof of Proposition \ref{prop:linear_variance}] By Corollary \ref{coro:var}, it remains to show that $\kappa>0$. 
Assume, for contradiction,  that  $\kappa=0$.  By \eqref{var-tildeY} and Chebyshev's inequality, for every $\varepsilon\in(0,1)$, when $n$ is large enough, 
\begin{align}\label{ev1}
\mathbb P(|\widetilde Y_n-\mathbb E[\widetilde Y_n]|<\varepsilon\sqrt n)>1-\varepsilon.
\end{align}

By Lemma \ref{lem:YN}, $\widetilde Y_{n+2N_n}=\widehat Y_n$, and
\[
\frac{N_n-\lambda n}{\sqrt{n}}\rightarrow\mathcal N(0,\sigma^2)
\]
is independent of $\widehat Y_n$, where the parameters $\lambda=\frac 1 {2(8d-1)}$ and {$\sigma^2=\frac{4d}{(8d-1)^2}$} can be easily determined.
By taking $\varepsilon$ small enough, for every $n$ large enough,
\begin{align}\label{ev2}
\mathbb P(N_n<\lambda n-\sqrt n)>3\sqrt{\varepsilon},
\end{align}
\begin{align}\label{ev3}
\mathbb P(N_n>\lambda n+\sqrt n)>3\sqrt{\varepsilon}.
\end{align}
For each $n$ large enough, with probability at least $3\sqrt{\varepsilon}-\varepsilon>2\sqrt{\varepsilon}$, \eqref{ev2} happens for $n$ and \eqref{ev1} happens for $n'=n+2\lambda n-2\sqrt n$ simultaneously, in other words,
\[
|\widetilde Y_{n+2\lambda n-2\sqrt n}-\mathbb E[\widetilde Y_{n+2\lambda n-2\sqrt n}]|<\varepsilon\sqrt {n+2\lambda n-2\sqrt n},
\quad
N_n<\lambda n-\sqrt n.
\]
Combine them and notice that $({\widetilde Y_n})$ is monotone increasing, we have
\[
\mathbb E[\widetilde Y_{n+2\lambda n-2\sqrt n}]+\varepsilon\sqrt {n+2\lambda n-2\sqrt n}>\widetilde Y_{n+2\lambda n-2\sqrt n}\ge \widetilde Y_{n+2N_n}=\widehat{Y}_n,
\]
i.e. for each large $n$,
\begin{align}\label{ev4}
\mathbb P(
\widehat{Y}_n<\mathbb E[\widetilde Y_{n+2\lambda n-2\sqrt n}]+\varepsilon(1+2\lambda)\sqrt n
)>2\sqrt\varepsilon.
\end{align}

Since $(N_n)$ is independent of $(\widehat{Y}_n)$, with probability at least $2\sqrt\varepsilon\cdot 3\sqrt\varepsilon=6\varepsilon$,
\eqref{ev3} and \eqref{ev4} happens simultaneously, i.e.
\[N_n>\lambda n+\sqrt n,\quad
\widehat{Y}_n<\mathbb E[\widetilde Y_{n+2\lambda n-2\sqrt n}]+\varepsilon(1+2\lambda)\sqrt n.
\]
So with probability at least $6\varepsilon$,
\begin{align}\label{ev5}
\widetilde Y_{n+2\lambda n+2\sqrt n}\le \widetilde Y_{n+2N_n}=\widehat{Y}_n<\mathbb E[\widetilde Y_{n+2\lambda n-2\sqrt n}]+\varepsilon(1+2\lambda)\sqrt n.
\end{align}

Now combine \eqref{ev1} and \eqref{ev5}, with probability at least $6\varepsilon-\varepsilon=5\varepsilon$, we have
\[
\mathbb E[\widetilde Y_{n+2\lambda n+2\sqrt n}]-\varepsilon \sqrt {n+2\lambda n+2\sqrt n}<\widetilde Y_{n+2\lambda n+2\sqrt n}<\mathbb E[\widetilde Y_{n+2\lambda n-2\sqrt n}]+\varepsilon (1+2\lambda)\sqrt n.
\]
Then we deduce that 
\begin{align}\label{ev6}
\mathbb E[\widetilde Y_{n+2\lambda n+2\sqrt n}]-\mathbb E[\widetilde Y_{n+2\lambda n-2\sqrt n}]<2\varepsilon (1+2\lambda)\sqrt n.
\end{align}
Now that \eqref{ev6} holds for all large enough $n$, together with monotonicity of $\widetilde Y_n$, we have
\[
\limsup_{n\rightarrow\infty}\frac{\mathbb E[\widetilde Y_n]}{n}\le\varepsilon(1+2\lambda)/2.
\]
Since $\varepsilon$ is arbitrary  and $0\le Y_n- \widetilde Y_n\le 1$,  this implies 
$\limsup_{n\rightarrow\infty}\frac{\mathbb E[Y_n]}{n}=\limsup_{n\rightarrow\infty}\frac{\mathbb E[\widetilde Y_n]}{n}=0$, which is absurd as long as $d\ge 5$ (by \cite{LeGall-Lin-range}).
\end{proof}

\section{Proofs of \texorpdfstring{\eqref{cond-a}  and  \eqref{cond-b}}{} }\label{s:conda+b}

Let
\[\F_i=\sigma(u_i,u_{i-1},\dots;V(i),V(i-1),\dots).\]
Recall that 
$Y_n=\xi_{1}^\infty+\dots+\xi^\infty_n$
and for every random variable $X$,
\[
\braces X=X-\mathbb E[X].
\]

We prove condition \eqref{cond-a} in the following lemma.

\begin{lemma}\label{lem:condition_1} Let $d>10$. As $n\to\infty$,  we have
\[
\frac 1{\sqrt n}\mathbb E\bracksof*{\braces{Y_n}} {\F_0}\overset{L^1}{\longrightarrow}0.
\]
\end{lemma}
\begin{proof}  Let $k_n= \lfloor n^\alpha\rfloor $ for $\alpha\in (\frac1{d-4}, \frac12)$. Recall \eqref{xi-ki} for $\xi_i^{k_n}$. By Lemma \ref{lem:xi_to_infty}, 
\begin{equation} \sum_{i=1}^n \e | \xi_i^\infty - \xi_i^{k_n}| \lesssim n ^{1 + (4-d)\alpha/2} = o(n^{1/2}). 
\label{eq:xi_change}
\end{equation} 
So it suffices to show that,
\[
\frac 1{\sqrt n}\sum_{i=1}^n \mathbb E\bracksof*{{\braces{\xi_i^{k_n}}}}{\F_0}\overset{L^1}{\longrightarrow}0.
\]
For each $0\le j\le n$, further denote
\[
Y_j^{(n)}:=\sum_{i=1}^n\braces{\xi_i^{k_n}\1_{\{u_i\in\mathcal T^+_j\}}},
\]
then by definition $\mathbb E[Y_j^{(n)}]=0$, and
$\sum_{i=1}^n \braces{\xi_i^{k_n}}=\sum_{j=0}^n Y_j^{(n)}.$

Let $\delta>0$ be small,  with its value to be determined later. Let $$\ell_n:= \lfloor n^{\frac12+\delta}\rfloor.$$ For $j\ge \ell_n$, observe that $u_0,\dots,u_n$ reaches $\T^+_{j}$ with small probability,
\begin{align}\label{eq:sqrt_n}
\p(\#\T^+_0+\dots+\#\T^+_{j-1}<n)
\lesssim \p(\#\T^+_0<n)^j\lesssim e^{-cn^\delta},
\end{align}

\noindent where $c$ is some positive constant and  the last inequality follows from the fact that $\p(\#\T^+_0 >n) \sim c n^{-1/2}$. Then, 
\[
\sum_{j>n^{1/2+\delta}}\e\absolute*{\e\bracksof{Y^{(n)}_j}{\F_0}}
\lesssim \sum_{i=0}^n\sum_{j>n^{1/2+\delta}}\p(u_i\in\T_j^{{+}})\lesssim ne^{-cn^\delta},
\]
so
\begin{align*}
\frac 1{\sqrt n}\sum_{j>\ell_n}\e\bracksof{Y_j^{(n)}}{\F_0}
\overset{L^1}{\rightarrow}0.
\end{align*}

Below we prove that
\begin{align}\label{eq:goal_Y}
\frac 1{\sqrt n}\sum_{j=0}^{\ell_n}\e\bracksof{Y_j^{(n)}}{\F_0}
\overset{L^1}{\rightarrow}0.
\end{align}
We regroup \eqref{eq:goal_Y}  into
\begin{align}\label{eq:Ai}
\sum_{j=0}^{\ell_n}\e\bracksof{Y^{(n)}_j}{\F_0}
=\sum_{i=0}^{L_n}\sum_{j=0}^{2k_n-1}\e\bracksof{Y^{(n)}_{2ik_n+j}}{\F_0}=:\sum_{i=0}^{L_n} A_i
=\sum_{i=0, i\, \mbox{\tiny even}}^{L_n} A_i+ \sum_{i=1, i\, \mbox{\tiny odd}}^{L_n} A_i,
\end{align}
where  $L_n:=\frac{\ell_n}{2k_n}$ is treated as integer  for simplicity.

Notice that for any $j\ge k_n$, $\mathbb E\bracksof{Y_j^{(n)}}{\F_0}$ only depends on $\sigma(\mathcal T^-_{j-k_n+1},\dots,\mathcal T^-_{j+k_n-1}),$
hence when $|j_1-j_2|\ge 2k_n$, $\mathbb E\bracksof{Y_{j_1}^{(n)}}{\F_0}$ and $\mathbb E\bracksof{Y_{j_2}^{(n)}}{\F_0}$ are independent.
Therefore,  $A_0,A_2,A_4,\dots$ are independent (and centered), we deduce that
\begin{align}\label{eq:A_i}
\e\bracks*{\pars*{\sum_{i=0, i\, \mbox{\tiny even}}^{L_n} A_i}^2}
=\sum_{i=0, i\, \mbox{\tiny even}}^{L_n} \e\bracks*{A_i^2}.
\end{align}

Arbitrarily fix $j$, notice that  $(\1_{\{u_i\in\mathcal T^+_j\}}, \xi_i^{k_n}\1_{\{\d(u_i,\Sp)>k_n\}})$ is independent of $\F_0$,  where $\Sp=\{\varnothing_0,\varnothing_1,\dots\}$ denotes the spine. We have
\[
\mathbb E\bracksof{\braces{\xi_i^{k_n}\1_{\{\d(u_i,\Sp)>k_n, u_i\in\mathcal T^+_j\}}}}{\F_0}=0,
\]

Therefore, 
\begin{align*}
|\mathbb E\bracksof{Y_j^{(n)}}{\F_0}|
=&\absolute*{\sum_{i=0}^n\mathbb E\bracksof{\braces{\xi_i^{k_n}\1_{\{\d(u_i,\Sp)\le k_n, u_i\in\mathcal T^+_j\}}}}{\F_0}}\\
\le&\sum_{i=0}^n\pars*{\mathbb E\bracksof{{\xi_i^{k_n}\1_{\{\d(u_i,\Sp)\le k_n, u_i\in\mathcal T^+_j\} }}}{\F_0} 
+ \mathbb E\bracks{\xi_i^{k_n}\1_{\{\d(u_i,\Sp)\le k_n, u_i\in\mathcal T^+_j\}}}} \\
\le&2\sum_{i=0}^n\p\pars{{\d(u_i,\Sp)\le k_n},{u_i\in\mathcal T^+_j}},
\end{align*}

\noindent where, in the last inequality, we use $\xi_i^{k_n}\le 1$ and the independence of $\{\d(u_i,\Sp)\le k_n, u_i\in\mathcal T^+_j\} $ and $\F_0$ to eliminate  the conditional expectation. Then $$ 
|\mathbb E\bracksof{Y_j^{(n)}}{\F_0}|
\le
  2\e\bracks{\#\setof{i\ge 0}{\d(u_i,\Sp)\le k_n,u_i\in\mathcal T^+_j}}
  \lesssim k_n.
$$

\noindent It follows that 
\[
\e\bracks*{A_i^2}=\e\bracks*{\pars*{\sum_{j=0}^{2k_n-1}\mathbb E\bracksof{Y_{2ik_n+j}^{(n)}}{\F_0}}^2}\lesssim k_n^4
\]
Together with \eqref{eq:A_i}, we have $$\e\Big| \sum_{i=0, i\, \mbox{\tiny even}}^{L_n} A_i\Big| \le \sqrt{\sum_{i=0, i\, \mbox{\tiny even}}^{L_n} \e\bracks*{A_i^2}}
\lesssim \sqrt{L_n k_n^4}=n^{\frac {1+2\delta+6\alpha} 4}.$$

\noindent The same estimate holds for $\sum_{i=1, i\, \mbox{\tiny odd}}^{L_n} A_i$. Hence 
\[
\e\absolute*{\sum_{i=0}^{L_n}A_{i}}\lesssim  n^{\frac {1+2\delta+6\alpha} 4}.
\]
For $d>10$, we may take $\alpha\in(\frac 1{d-4},\frac16)$, and $\delta \in (0, (1-6\alpha)/2)$. Then $\frac {1+2\delta+6\alpha} 4< \frac12$, and we have \eqref{eq:goal_Y}. 
\end{proof}

We now prove condition \eqref{cond-b}. By definition, $\kappa$ is the limit of $\frac{\Var(Y_n)}{n}$, so it suffices to show that

\begin{lemma} \label{Lemma:cond-c} Let {$d>16$.} We have 
\begin{equation}    
\frac 1 n\pars*{\mathbb E\bracksof*{{\braces{Y_n}^2}}{\F_0}
-\mathbb E\bracks*{{\braces{Y_n}^2}}}\overset{L^1}{\longrightarrow}0. \label{cond-c-1}
\end{equation}
\end{lemma}
\begin{proof}

Let $k_n= \lfloor n^\alpha\rfloor $ for $\alpha\in (\frac1{d-4}, \frac12)$.  {
We first show that
\begin{align}\label{eq:claim_2nd}
\e \Big[\Big( Y_n - \sum_{i=1}^n \xi_i^{k_n}\Big)^2\Big] = o(n). 
\end{align}

Indeed, 
\begin{align*}
 \e \Big[\Big( Y_n - \sum_{i=1}^n \xi_i^{k_n}\Big)^2\Big] 
=&\e \Big[\Big( \sum_{i=1}^n (\xi_i^{k_n}-\xi_i^{k_n^2}+\xi_i^{k_n^2}-\xi_i^{\infty})\Big)^2\Big]\\
\le & 2 \e \Big[\Big( \sum_{i=1}^n (\xi_i^{k_n}-\xi_i^{k_n^2})\Big)^2\Big]
+2 \e \Big[\Big( \sum_{i=1}^n (\xi_i^{k_n^2}-\xi_i^{\infty})\Big)^2\Big].
 \end{align*}

By Lemma \ref{lem:xi_to_infty}, the second term is
\[
\e \Big[\Big( \sum_{i=1}^n (\xi_i^{k_n^2}-\xi_i^{\infty})\Big)^2\Big]
\le n\e\Big[\sum_{i=1}^n (\xi_i^{k_n^2}-\xi_i^{\infty})\Big]
\lesssim n^2k_n^{4-d}=n^{2+(4-d)\alpha}=o(n).
\]

For the first term, notice that $0\le \xi_i^{k_n}-\xi_i^{k_n^2}\le 1$ for any $i$, hence
\begin{align*}
&\e \Big[\Big( \sum_{i=1}^n (\xi_i^{k_n}-\xi_i^{k_n^2})\Big)^2\Big]\\
\le&2\sum_{1\le i\le j\le n}\e \Big[(\xi_i^{k_n}-\xi_i^{k_n^2})(\xi_j^{k_n}-\xi_j^{k_n^2})\Big]\\
\le&2\sum_{1\le i\le j\le n}\e \Big[(\xi_i^{k_n}-\xi_i^{k_n^2})\cdot \1_{\{\d(i,j)>2k_n^2\}}(\xi_j^{k_n}-\xi_j^{k_n^2})\Big]
+2\sum_{1\le i\le n}\e \Big[(\xi_i^{k_n}-\xi_i^{k_n^2})\cdot \sum_{i\le j\le n}\1_{\{\d(i,j)\le 2k_n^2\}}\Big] .  
\end{align*}

Applying  Lemma \ref{lem:independent_xi} twice,  \begin{align*} \e \Big[(\xi_i^{k_n}-\xi_i^{k_n^2})\cdot \1_{\{\d(i,j)>2k_n^2\}}(\xi_j^{k_n}-\xi_j^{k_n^2})\Big]
&= \e \Big[\xi_i^{k_n}-\xi_i^{k_n^2}\Big] \e\Big[\1_{\{\d(i,j)>2k_n^2\}}(\xi_j^{k_n}-\xi_j^{k_n^2})\Big]
\\
&\le \e \Big[\xi_i^{k_n}-\xi_i^{k_n^2}\Big] \e \Big[\xi_j^{k_n}-\xi_j^{k_n^2}\Big] \\
& \le \e \Big[\xi_i^{k_n}-\xi_i^\infty\Big] \e \Big[\xi_j^{k_n}-\xi_j^\infty\Big]
. \end{align*}

For any $i\le n$, by the translational invariance, \begin{align*} \e\Big[(\xi_i^{k_n}-\xi_i^{k_n^2})\cdot \sum_{i\le j\le n}\1_{\{\d(i,j)\le 2k_n^2\}}\Big] 
&= \e\Big[(\xi_0^{k_n}-\xi_0^{k_n^2})\cdot \sum_{0\le j\le n-i}\1_{\{\d(0,j)\le 2k_n^2\}}\Big]
\\
&=\e\Big[\xi_0^{k_n}-\xi_0^{k_n^2}\Big] \, \e\Big[  \sum_{0\le j\le n-i}\1_{\{\d(0,j)\le 2k_n^2\}}\Big]
\\
&\le \e\Big[\xi_0^{k_n}-\xi_0^\infty\Big] \, \e\Big[  \sum_{0\le j\le n-i}\1_{\{\d(0,j)\le 2k_n^2\}}\Big], \end{align*}

\noindent where the second equality follows from the independence of $\T^-$ and $\T^+$. Therefore, we have \begin{align*}
&\e \Big[\Big( \sum_{i=1}^n (\xi_i^{k_n}-\xi_i^{k_n^2})\Big)^2\Big]\\ 
\le&2\sum_{1\le i\le j\le n}\e \Big[\xi_i^{k_n}-\xi_i^\infty\Big]\e \Big[\xi_j^{k_n}-\xi_j^\infty\Big]
+2(n+1)\e \Big[\xi_0^{k_n}-\xi_0^\infty\Big]\cdot \e\Big[\sum_{0\le j\le n}\1_{\{\d(0,j)\le 2k_n^2\}}\Big].
\end{align*}

\noindent Now, we apply Lemma \ref{lem:xi_to_infty}  and \eqref{eq:moment_geo1}    to obtain that ($d>12$ suffices here)
\begin{align*}
&\e \Big[\Big( \sum_{i=1}^n (\xi_i^{k_n}-\xi_i^{k_n^2})\Big)^2\Big]
\lesssim n^2k_n^{4-d}+nk_n^{\frac{4-d}{2}}\cdot k_n^4=o(n).
\end{align*}
This finishes the proof of \eqref{eq:claim_2nd}.}

Therefore,  the proof of \eqref{cond-c-1} reduces to showing that \begin{equation}    
\frac 1 n\pars*{\mathbb E\bracksof*{{\braces{\sum_{i=0}^n \xi_i^{k_n}}^2}}{\F_0}
-\mathbb E\bracks*{{\braces{\sum_{i=0}^n \xi_i^{k_n}}^2}}}\overset{L^1}{\longrightarrow}0. \label{cond-c-2}
\end{equation}

We next make a further simplification of \eqref{cond-c-2}. Let
\[
U_j^{(n)}:=\sum_{i=0}^n \xi_i^{k_n}\1_{\{u_i\in\mathcal T^+_j,\d(u_i,\Sp)\le k_n\}},
\qquad
W_n:=\sum_{i=0}^n \xi_i^{k_n}\1_{\{\d(u_i,\Sp)>k_n\}}.
\]

We have $\sum_{i=0}^n \xi_i^{k_n}= \sum_{j=0}^n U_j^{(n)} + W_n$.  Let $\delta >0$ be small, and set  as in the proof of Lemma \ref{lem:condition_1}, 
$\ell_n:= \lfloor n^{\frac12+\delta}\rfloor$, 
\begin{equation}  \Upsilon_n:= \sum_{j=0}^{\ell_n}\braces{U_j^{(n)}}+\braces{W_n} .   
\label{eq:Y=U+W} \end{equation} 
By \eqref{eq:sqrt_n}, we see that 
\begin{align}\label{eq:add_l}
\e\Big|\sum_{j=\ell_n+1}^n\braces{U_j^{(n)}}\Big|
\lesssim n^2\p\pars{\#\T^+_0+\dots+\#\T^+_{\ell_n} \le n} =o(1),
\end{align}
so the proof of \eqref{cond-c-2} reduces to showing that 
\begin{equation}   
\frac 1 n \left( \e \Big[\Upsilon_n^2 \, \big|\, \F_0\Big]- 
\e \Big[\Upsilon_n^2 \Big]\right)  \overset{L^1}{\longrightarrow}0.\label{cond-c-3}
\end{equation}

Since $W_n$ is independent of $\F_0$, we have
\[  \mathbb E\bracksof*{{\braces{W_n}^2}}{\F_0}
-\mathbb E\bracks*{{\braces{W_n}^2}} =0,
\]
it is enough to  show that
\begin{align}\label{eq:first}
\frac{1}{n}\Big(\e\Big[\Big(\sum_{j=0}^{\ell_n}\braces{U_j^{(n)}}\Big)^2 \big| \F_0\Big]-{\e\Big[\Big(\sum_{j=0}^{\ell_n}\braces{U_j^{(n)}}\Big)^2  \Big]}\Big)\overset{L^1}{\longrightarrow}0,
\end{align}
and
\begin{align}\label{eq:second}
\frac{1}{n} \Big(\e\Big[\braces{W_n}\sum_{j=0}^{\ell_n}\braces{U_j^{(n)}} \,\big|\, \F_0\Big]
-\e\Big[\braces{W_n}\sum_{j=0}^{\ell_n}\braces{U_j^{(n)}} \Big] \Big)\overset{L^1}{\longrightarrow}0.
\end{align}

The proof of \eqref{eq:first}  is similar to that of Lemma \ref{lem:condition_1}. In fact, write
\begin{equation}    
\e\Big[\Big(\sum_{j=0}^{\ell_n}\braces{U_j^{(n)}}\Big)^2\big|\F_0\Big]=
\sum_{j,j'=0}^{\ell_n}\mathbb E\bracksof{\braces{U_j^{(n)}}\braces{U_{j'}^{(n)}}}{\F_0}=:\sum_{j,j'=0}^{\ell_n}\Theta_{j,j'}^{(n)}. \label{eq:new2}
\end{equation}

{ \noindent We are going to show that \begin{equation}\label{eq:new1} \frac1{n} \e\big|\sum_{j,j'=0}^{\ell_n}\langle\Theta_{j,j'}^{(n)}\rangle\big|  \to 0, \qquad n \to \infty.   \end{equation}
}

To this end, note that  $0\le \xi^{k_n}_i\le 1$, and $\1_{\{u_i\in\T^+_j,\d(u_i,\Sp)\le k_n\}}$ is independent of $\F_0$. We have for any $i, i'\ge 0$, 
\begin{align*}
&\absolute*{\e\bracksof*{\braces{\xi_i^{k_n}\1_{\{u_i\in\mathcal T^+_j,\d(u_i,\Sp)\le k_n\}}}\braces{\xi_{i'}^{k_n}\1_{\{u_{i'}\in\mathcal T^+_{j'},\d(u_{{i'}},\Sp)\le k_n\}}}}{\F_0}}\\
\le
  &  \e\bracks*{{\1_{\{u_i\in\mathcal T^+_j,\d(u_i,\Sp)\le k_n\}}}{\1_{\{u_{i'}\in\mathcal T^+_{j'},\d(u_{{i'}},\Sp)\le k_n\}}}}+
2 \e\bracks*{{\1_{\{u_i\in\mathcal T^+_j,\d(u_i,\Sp)\le k_n\}}}}
\e\bracks*{\1_{\{u_{i'}\in\mathcal T^+_{j'},\d(u_{{i'}},\Sp)\le k_n\}}}.
\end{align*}

Summing over $0\le i,i'\le n$, we obtain that \begin{align*} 
|\Theta^{(n)}_{j,j'}|
& \le
\e \Big[ \big(\sum_{i=0}^n \1_{\{u_i\in\mathcal T^+_j,\d(u_i,\Sp)\le k_n\}} \big)  \big(\sum_{i=0}^n \1_{\{u_i\in\mathcal T^+_{j'},\d(u_i,\Sp)\le k_n\}} \big) \Big]+  \\
& \qquad 2 \e \Big[  \sum_{i=0}^n \1_{\{u_i\in\mathcal T^+_j,\d(u_i,\Sp)\le k_n\}}  \Big] \e \Big[ \sum_{i=0}^n \1_{\{u_i\in\mathcal T^+_{j'},\d(u_i,\Sp)\le k_n\}}   \Big]. \end{align*}

Since $(\T^+_j, \T^+_{j'})$ have the same law as $(\T^-_j, \T^-_{j'})$, if we continue to use the notation $Z_\ell^{(j)}$ for the population of the $\ell$-th generation of $\T^-_j$, then $$|\Theta^{(n)}_{j,j'}|
  \le
\e \Big[ \big(\sum_{\ell=0}^{k_n} Z_\ell^{(j)}\big)  \big(\sum_{\ell=0}^{k_n} Z_\ell^{(j')}\big) \Big]+   2 (1+k_n)^2. $$

{By \eqref{eta-pmoment}},  we obtain
\begin{equation}   
|\Theta^{(n)}_{j,j'}|\lesssim k_n^2+k_n^3\1_{\{j=j'\}}. \label{Theta-1}
 \end{equation}

Since
$\Theta_{j,j'}^{(n)}$ is measurable with respect to $\sigma(\mathcal T^-_{j-k_n+1},\dots,\mathcal T^-_{j+k_n-1},\mathcal T^-_{j'-k_n+1},\dots,\mathcal T^-_{j'+k_n-1})$, we deduce that
$\Theta_{j,j'}^{(n)},\Theta_{m,m'}^{(n)}$ are independent if $\min(|j-m|,|j'-m|,|j-m'|,|j'-m'|)\ge 2k_n$. { In particular, $\Cov(\Theta^{(n)}_{j,j'},\Theta^{(n)}_{m, m'})$ vanishes unless at least one of the four distances $|j-m|,|j-m'|,|j'-m|,|j'-m'|$ is smaller than $2k_n$; the number of such quadruples in $\{0,\dots,\ell_n\}^4$ is $O(\ell_n^3k_n)$. By \eqref{Theta-1} each covariance is $O(k_n^4)$ off the diagonals, and the quadruples involving a diagonal pair $j=j'$ or $m=m'$ contribute $O(\ell_n^2k_n^6+\ell_nk_n^7)$, which is smaller. Consequently, 
\begin{align*}    
\mbox{Var}\Big(\sum_{j,j'=0}^{\ell_n}\Theta_{j,j'}^{(n)}\Big) 
&\lesssim (\ell_n^2 k_n^6 + \ell_n k_n^7) + \ell_n^3 k_n \times k_n^4
 \lesssim \ell_n^3 k_n^5.
\end{align*}
Hence
\begin{equation*} 
 \e\big|\sum_{j,j'=0}^{\ell_n}\langle\Theta_{j,j'}^{(n)}\rangle\big|
 \le\sqrt{\mbox{Var}\Big(\sum_{j,j'=0}^{\ell_n}\Theta_{j,j'}^{(n)}\Big)} \lesssim\sqrt{\ell_n^3k_n^5}=n^{\frac34+\frac32\delta+\frac52\alpha}. 
\end{equation*}
 }

Since {$\alpha \in (\frac1{d-4}, \frac12)$, if $d>14$,} we may take some $\alpha$ sufficiently close to 
$\frac1{d-4}$ and $\delta$ sufficiently small such that {$\frac34+\frac32\delta+\frac52\alpha < 1$.} This yields \eqref{eq:new1}, hence \eqref{eq:first}.

\medskip
Now, we prove \eqref{eq:second}.
Denote
\[
\Xi_j:=\mathbb E\bracksof*{\braces{W_n}\braces{U_j^{(n)}}}{\F_0},
\]
we need to prove that \begin{equation}    \frac{1}{n}\mathbb E\absolute*{\sum_{j=0}^{\ell_n}\braces{\Xi_j}}{\rightarrow}0. \label{eq:second-2} \end{equation}

Observe that
$\Xi_j$ is measurable with respect to $\sigma(\mathcal T^-_{j-k_n+1},\dots,\mathcal T^-_{j+k_n-1})$, so
 $\Xi_i$,$\Xi_j$ are independent when $|i-j|\ge 2k_n$.  Let $L_n=\frac{\ell_n}{2k_n}$. Apply the same trick as in \eqref{eq:Ai} again, we can regroup the sum \begin{equation}    \sum_{j=0}^{\ell_n}\braces{\Xi_j}= \sum_{i=0}^{L_n} \, \Big(\sum_{j=0}^{2k_n-1}\braces{\Xi_{2ik_n+j}}\Big)=: \sum_{i=0}^{L_n} \, D_i.  
 \label{sumDi} \end{equation}

 Assume for the moment that for   $\gamma' := 5\alpha + 2 \delta$, \begin{equation} \max_{0\le j \le \ell_n} \e \big[ \braces{\Xi_j}^2\big] \lesssim n^{ 1+   \gamma'}, \label{maxXi-j}  \end{equation}
 
 \noindent then the triangular inequality for the $L^2$-norm yields that  \begin{equation}    \max_{0\le i \le L_n} \e [D_i^2] \lesssim k_n^2 n^{ 1+    \gamma'}. \label{maxDi-L2} \end{equation}

Note that $D_i$ and $D_{i'}$ are independent if $|i-i'| \ge 2$. By decomposing the sum $\sum_{i=0}^{L_n} \, D_i$ into two parts, according to the parity of the index $i$, we deduce from  \eqref{maxDi-L2} that $$ \e\Big| \sum_{j=0}^{\ell_n}\braces{\Xi_j} \Big| \lesssim \sqrt{L_n k_n^2 n^{ 1+   \gamma'}} \le n^{ \frac34+ \frac{\gamma'}2+\frac\delta2+ \frac\alpha2} . $$

 \noindent Since $\alpha$ can be chosen arbitrarily close to {$\frac1{d-4}$, for $d>16$, }  we can choose $\alpha < \frac1{12}$, $\delta$ sufficiently small so that $\frac34+ \frac{\gamma'}2+\frac\delta2+ \frac\alpha2< 1$. Hence  \eqref{eq:second-2} holds.

 It remains to show \eqref{maxXi-j}. {By the conditional Cauchy-Schwarz inequality and the fact that $\braces{W_n}^2$ is independent of $\F_0$,  for any $0\le j\le \ell_n$, $$
\mathbb E\bracks*{\braces{\Xi_{j}}^2} 
\le \e [\Xi_j^2]
\le \e[\braces{W_n}^2] \, \e[\braces{U_j^{(n)}}^2].$$

 Since $\ell^2_n k_n^5= n^{1+\gamma'}$,   \eqref{maxXi-j} will follow once we prove the following two estimates: 
 \begin{align} \e[\braces{W_n}^2] & \lesssim \ell^2_n k_n^2, \label{eq:Wn2}
 \\
 \max_{0\le j\le \ell_n} \e[\braces{U_j^{(n)}}^2] & \lesssim k_n^3. \label{eq:Ujn2} \end{align}

The proof of \eqref{eq:Ujn2} is easy. Let $j\le \ell_n$. Denote 
\[\Gamma_j=\#\setof{i\ge 0}{u_i\in\mathcal T^+_j,\d(u_i,\Sp)\le k_n},
\]
then ${U_j^{(n)}}\le\Gamma_j$, so $\mathbb E[\braces{U_j^{(n)}}^2]\le\mathbb E[\Gamma_j^2]$. Each $\Gamma_j$ has the same distribution as $\Gamma_0$ which is the total progeny of a critical Galton-Watson process up to $k_n$-th generation. Applying \eqref{eta-pmoment} to $k_n$ in lieu of $n$, and $p=2$, we get that $\e[\Gamma_j^2]\lesssim k_n^3$, implying  \eqref{eq:Ujn2}. 

Finally we  prove \eqref{eq:Wn2}.   By reverting \eqref{eq:Y=U+W},   
 $\braces{W_n}=\braces{\Upsilon_n}-\sum_{j=0}^{\ell_n}\braces{U_j^{(n)}}.$  Then 
 $$ \e[\braces{W_n}^2]\le 2 \e [\braces{\Upsilon_n}^2] + 2 \e \Big[\Big(\sum_{j=0}^{\ell_n}\braces{U_j^{(n)}}\Big)^2\Big].$$
 
\noindent 
 By definition, $\Upsilon_n=\braces{\sum_{i=0}^n\xi^{k_n}_i}-\sum_{j=\ell_n+1}^n\braces{U^{(n)}_j}$. The term $\big|\sum_{j=\ell_n+1}^n\braces{U^{(n)}_j}\big|\le n$ a.s., therefore by \eqref{eq:add_l}, $\mathbb E[\braces{\sum_{j=\ell_n+1}^n\braces{U^{(n)}_j}}^2] \lesssim n$.    By the already proved \eqref{eq:claim_2nd} and \eqref{kappaexistspositive},  $\e[\braces{\sum_{i=0}^n\xi^{k_n}_i}^2] \lesssim n$, then we have shown that $$ \mathbb E[\braces{\Upsilon_n}^2] \lesssim n.$$
 
It remains to bound the second moment of $\sum_{j=0}^{\ell_n}\braces{U_j^{(n)}}$. We  take   expectations on \eqref{eq:new2} and  then apply \eqref{Theta-1},  obtaining  $$ \e\Big(\sum_{j=0}^{\ell_n}\braces{U_j^{(n)}}\Big)^2 = \sum_{j, j'=0}^{\ell_n} \e[\Theta_{j, j'}^{(n)}]
   \lesssim \ell_n k_n^3 + \ell_n^2 k_n^2
   \lesssim \ell_n^2 k_n^2,$$ 
   
 \noindent implying \eqref{eq:Wn2}. This completes the proof of \eqref{maxXi-j}  and  hence that of Lemma \ref{Lemma:cond-c}.  }
\end{proof}

\section{Proof of \texorpdfstring{\eqref{cond-c}}{}: Uniform integrability by fourth moment estimate}\label{s:cond-c}

In this section, we prove that 
\begin{proposition}\label{P:4thmoment} Assume that $d> 16$. We have that for all $n\ge 1$, 
\[
\mathbb E[\braces{Y_n}^4]\lesssim n^2.
\]
In particular, this implies that $(\frac{\braces{Y_n}^2}{n})_{n\ge 1} $ is uniformly integrable.
\end{proposition}

Let
\begin{equation}   \Theta_n:=\e [ \langle Y_n\rangle^4] = \sum_{1\le \a, \b,\c,\t\le n}\mathbb E[\braces{\xi_\a^\infty}\braces{\xi_\b^\infty}\braces{\xi_\c^\infty}\braces{\xi_\t^\infty}]. \label{def-Theta_n}
\end{equation}

By symmetry, the sum  $\sum_{\a, \b, \c, \t}$ on the right-hand-side of   \eqref{def-Theta_n}   equals 
$$\sum_{\a, \b, \c, \t}=4  \sum_{1\le \a < \b\wedge \c \wedge \t} + 6 \sum_{1\le \a=\b < \c\wedge \t} +  4 \sum_{\a=\b=\c < \t} + \sum_{\a=\b=\c=\t},$$
where in the above four sums, the indices $\a, \b, \c, \t$ run from $1$ to $n$  and   $i\wedge j:= \min(i, j)$ for any $i, j\in \z$. 

Using $|\braces{\xi_i^\infty}|\le 1 $ for any $i$, we see that $|\sum_{\a=\b=\c < \t} \mathbb E[\braces{\xi_\a^\infty}\braces{\xi_\b^\infty}\braces{\xi_\c^\infty}\braces{\xi_\t^\infty}]|\lesssim n^2$ and $|\sum_{\a=\b=\c=\t} \mathbb E[\braces{\xi_\a^\infty}\braces{\xi_\b^\infty}\braces{\xi_\c^\infty}\braces{\xi_\t^\infty}]|\lesssim n$. Furthermore, 
{
\begin{align*}
\sum_{1\le \a=\b < \c\wedge \t} \mathbb E[\braces{\xi_\a^\infty}\braces{\xi_\b^\infty}\braces{\xi_\c^\infty}\braces{\xi_\t^\infty}]
=&\sum_{\a=1}^{n-1}\sum_{\c,\t=\a+1}^n
\mathbb E[\braces{\xi_\a^\infty}^2\braces{\xi_\c^\infty}\braces{\xi_\t^\infty}]\\
=&\sum_{\a=1}^{n-1}
\mathbb E\bracks*{
\braces{\xi_\a^\infty}^2\cdot 
\pars*{\sum_{\c=\a+1}^n\braces{\xi_\c^\infty}}^2}\\
\le&
\sum_{\a=1}^{n-1}\mathbb E\bracks*{
\pars*{\sum_{\c=\a+1}^n\braces{\xi_\c^\infty}}^2}
=\sum_{\a=1}^{n-1}\mathbb E\bracks*{\braces{Y_{n-a}}^2}\lesssim n^2,
\end{align*}

\noindent where the last inequality follows from Proposition \ref{prop:linear_variance}.}

By translational invariance (see \eqref{eq:invariance}), $$\sum_{1\le \a < \b\wedge \c \wedge \t} \mathbb E[\braces{\xi_\a^\infty}\braces{\xi_\b^\infty}\braces{\xi_\c^\infty}\braces{\xi_\t^\infty}] 
= \sum_{\a=1}^{n-1} \sum_{\b, \c, \t=1}^{n-a}  \mathbb E[\braces{\xi_0^\infty}\braces{\xi_\b^\infty}\braces{\xi_\c^\infty}\braces{\xi_\t^\infty}] 
=: \sum_{\a=1}^{n-1} J_{n-\a}.$$

Consequently, we have that  for all $n\ge 1$,  \begin{equation} \Theta_n - {4} \sum_{i=1}^{n-1} J_i \lesssim n^2. \label{Theta<sumJ}  \end{equation}
The main technical estimate in the proof of Proposition \ref{P:4thmoment} is the following lemma:

\begin{lemma}  \label{lem:main4th}  Let $d >16$. We have for all $n\ge 1$, $$ J_n  \lesssim \,  n + \sqrt{\Theta_n}.$$ 

\end{lemma}

Assume for the moment Lemma \ref{lem:main4th}, we can give the proof of Proposition \ref{P:4thmoment}:

\begin{proof}[Proof of Proposition \ref{P:4thmoment}.] By \eqref{Theta<sumJ}  and Lemma \ref{lem:main4th}, there is a constant $c>1$ such that for all $n\ge 1$, 
$$ \Theta_n \le c \sum_{i=1}^n \sqrt{\Theta_i}+ c n^2.$$ Let $x_n:=\max_{1\le i \le n} \Theta_i$. Then the above inequality implies that $x_n \le c \sum_{i=1}^n \sqrt{x_i}+ c n^2 \le c n \sqrt{x_n} + c n^2$. 
This implies $x_n \le 4 c^2 n^2,$ showing Proposition \ref{P:4thmoment}.  \end{proof}

The rest of this section is to show Lemma \ref{lem:main4th}. 
Denote $\alpha_n:=\lfloor\frac{\d(u_n,u_0)}2\rfloor$, and we divide $J_n$ into several parts. Let  \begin{align}
\xifour_1&:= \sum_{\a,\b,\c=1}^{n} \e \left[\braces{\xi_0^\infty} \braces{\xi_\a^\infty - \xi_\a^{\alpha_\a} }\braces{\xi_\b^\infty} \braces{\xi_\c^\infty} \right] ,  \label{def-I1}\\
\xifour_2&:=\sum_{\a,\b,\c=1}^{n} \e \left[\braces{\xi_0^\infty} \braces{\xi_\a^{\alpha_\a} }\braces{\xi_\b^\infty-\xi_\b^{\alpha_\b}} \braces{\xi_\c^\infty} \right],\label{def-I2} \\
\xifour_3&:=\sum_{\a,\b,\c=1}^{n} \e \left[\braces{\xi_0^\infty} \braces{\xi_\a^{\alpha_\a} }\braces{\xi_\b^{\alpha_\b}} \braces{\xi_\c^\infty-\xi_\c^{\alpha_\c}} \right], \label{def-I3}
\end{align}
\begin{align}
\xifour_4&:=\sum_{\a,\b,\c=1}^{n}\sum_{k_1,k_2,k_3=0}^\infty \e \left[\braces{\xi_0^\infty- \xi_0^{k_1\wedge k_2\wedge k_3}} \braces{\xi_\a^{k_1} \1_{\{\alpha_\a=k_1\}}}\braces{\xi_\b^{k_2} \1_{\{\alpha_\b=k_2\}}} \braces{\xi_\c^{k_3} \1_{\{\alpha_\c=k_3\}}} \right]. \label{def-I4}
\end{align}

\begin{lemma}\label{lem:I0} We have
\[
J_n= \xifour_1+\xifour_2+\xifour_3+\xifour_4 .
\]
\end{lemma}
\begin{proof}
First, we observe that
\begin{align*}
J_n-I_1-I_2-I_3
&=\sum_{\a,\b,\c=1}^n\mathbb E[\braces{\xi_0^\infty}\braces{\xi_\a^{\alpha_\a}}\braces{\xi_\b^{\alpha_\b}}\braces{\xi_\c^{\alpha_\c}}]\\
&=\sum_{\a,\b,\c=1}^n\sum_{k_1,k_2,k_3=0}^\infty
\mathbb E[\braces{\xi_0^\infty}\braces{\xi_\a^{k_1} \1_{\{\alpha_\a=k_1\}}}\braces{\xi_\b^{k_2} \1_{\{\alpha_\b=k_2\}}} \braces{\xi_\c^{k_3} \1_{\{\alpha_\c=k_3\}}}]\\
&=I_4+\sum_{\a,\b,\c=1}^{n}\sum_{k_1,k_2,k_3=0}^\infty\mathbb E[\braces{\xi_0^{k_1\wedge k_2\wedge k_3}}\braces{\xi_\a^{k_1}\1_{\{\alpha_\a=k_1\}}}\braces{\xi_\b^{k_2}\1_{\{\alpha_\b=k_2\}}}\braces{\xi_\c^{{k_3}}\1_{\{\alpha_\c=k_3\}}}].
\end{align*}

By Lemma \ref{lem:independent_xi},  
$\xi_0^{k_1\wedge k_2\wedge k_3}$ is independent of $(\xi_\a^{k_1}\1_{\{\alpha_\a=k_1\}}, \xi_\b^{k_2}\1_{\{\alpha_\b=k_2\}}, \xi_\c^{k_3}\1_{\{\alpha_\c=k_3\}})$, so
\begin{align*}
	&\mathbb E[\braces{\xi_0^{k_1\wedge k_2\wedge k_3}}\braces{\xi_\a^{k_1}\1_{\{\alpha_\a=k_1\}}}\braces{\xi_\b^{k_2}\1_{\{\alpha_\b=k_2\}}}\braces{\xi_\c^{k_3}\1_{\{\alpha_\c=k_3\}}}]\\
	=&\mathbb E[\braces{\xi_0^{k_1\wedge k_2\wedge k_3}}]\mathbb E[\braces{\xi_\a^{k_1}\1_{\{\alpha_\a=k_1\}}}\braces{\xi_\b^{k_2}\1_{\{\alpha_\b=k_2\}}}\braces{\xi_\c^{k_3}\1_{\{\alpha_\c=k_3\}}}]=0,
\end{align*}
proving the lemma.\end{proof}

Below we shall estimate $\xifour_1,\xifour_2,\xifour_3,\xifour_4$ separately. First, we present a technical lemma:

\begin{lemma} \label{lem:xi-alpha}  Let $d>4$. For any $0<\gamma < \frac{d-4}2$ and $m\ge 1$, there is some positive constant $c=c_{\gamma, m}$ only depending on $\gamma$ and $m$,   such that for all  $k\ge 1$ \begin{align} & \e \Big[\big( \sum_{i=0}^\infty(\xi_i^{k}-{\xi_i^\infty}) \1_{\{\d(u_0,u_i)\le m k\}} \big)^2\Big]
  \le c \,   k^{4-\gamma}. \label{eq:xi-alpha1}
   \end{align}
   Consequently, when $d>16$, \begin{align} & \e \Big[\Big( \sum_{i=0}^\infty(\xi_i^{\alpha_i}-{\xi_i^\infty})\Big)^2\Big] < \infty, \label{eq:xi-alphaL2} 
 \\ &
   \lim_{n\to{\infty}} \frac1{n} \e \Big[\Big( \sum_{i=0}^n \braces{\xi_i^{\alpha_i}}\Big)^2\Big] = \kappa, \label{eq:xi-alphaL2-bis}  
   \end{align} where we recall that $\alpha_i:= \lfloor\frac{\d(u_i,u_0)}2\rfloor$ and $\kappa$ is the constant in Corollary \ref{coro:var}.  \end{lemma}
   
 \begin{proof}  Let $n\ge 1$ and denote
\begin{align}
\Gamma_m(k, n):=
&\e \Big[\big( \sum_{i=0}^n(\xi_i^{k}-{\xi_i^\infty}) \1_{\{\d(u_0,u_i)\le m  k\}} \big)^2\Big] \label{def-Gamma-mk}\\
=&\sum_{0\le i,j\le n}\e \bracks*{ (\xi_i^{k}-{\xi_i^\infty})(\xi_j^{k}-{\xi_j^\infty}) \1_{\{\d(u_0,u_i),\d(u_0,u_j)\le mk\}} } \nonumber \\
\le&2\sum_{0\le i\le j\le n}\e \bracks*{ (\xi_i^{k}-{\xi_i^\infty})(\xi_j^{k}-{\xi_j^\infty}) \1_{\{\d(u_0,u_i)\le mk,\d(u_i,u_j)\le 2mk\}} }. \nonumber
\end{align}
By \eqref{eq:invariance}, move $(u_0,u_i,u_j)$ to $(u_{-i},u_0,u_{j-i})$ and denote $\ell=j-i$, we have
\begin{align*}
\Gamma_m{(k,n)}
\le&2\sum_{0\le i, \ell\le n}\e \bracks*{ (\xi_0^{k}-{\xi_0^\infty})(\xi_\ell^{k}-{\xi_\ell^\infty}) \1_{\{\d(u_{-i},u_0)\le mk,\d(u_0,u_\ell)\le 2mk\}} }\\
=&2\e \Big[ (\xi_0^{k}-{\xi_0^\infty})\cdot\sum_{i=0}^n\1_{\{\d(u_{-i},u_0)\le mk\}}\cdot \sum_{\ell =0}^n(\xi_\ell^{k}-{\xi_\ell^\infty}) \1_{\{\d(u_0,u_\ell)\le 2mk\}} \Big].
\end{align*}
Note that for nonnegative random variables $X,Y,Z$ and every integer $N>2$, 
\begin{align}\label{eq:Cauchy+}
\e[XYZ]\le \|X\|_{\frac{2N}{N-2}}\|Y\|_{N}\|Z\|_2,\end{align} 

\noindent where for any $p\ge 1$, $\|\cdot\|_p$ denotes the usual $L^p$-norm. 
Apply this and the fact that $\xi_0^{k}-{\xi_0^\infty}\in\{0,1\}$, we have
\begin{align*}
\Gamma_m{(k,n)}
\le&2\e \bracks*{\xi_0^{k}-{\xi_0^\infty}}^{\frac{N-2}{2N}}\cdot\e \Big[\big(\sum_{\a=1}^n\1_{\{\d(u_{-a},u_0)\le mk\}}\big)^N\Big]^{\frac 1 N}\cdot \e \Big[\big(\sum_{c=0}^n(\xi_\c^{k}-{\xi_\c^\infty}) \1_{\{\d(u_0,u_c)\le 2mk\}}\big)^2\Big]^{\frac 1 2}
\\
\le &
c' k^{\frac{4-d}{2}\frac{N-2}{2N}}\cdot(mk)^2\cdot\sqrt{\Gamma_{2m}{(k,n)}},
\end{align*}
where 
$c'$ is some positive constant depending only on $N$, and  
we use Lemma \ref{lem:xi_to_infty} and   \eqref{eq:moment_geo1} in the last line to bound $\e \bracks*{\xi_0^{k}-{\xi_0^\infty}}$ and the $N$-th moment, uniformly in $n$, of $\sum_{\a=1}^n\1_{\{\d(u_{-a},u_0)\le mk\}}$, respectively. We choose and then fix $N$ large enough such that $(d-4) (N-2)/N > {2 \gamma}$, so that   for any $m\ge 1$, 
$$ \Gamma_m{(k,n)} \le \lambda \, m^2 \, \sqrt{\Gamma_{2m}{(k,n)}},$$ 
with $\lambda=\lambda_k:= c'\, k^{2- \frac\gamma2}.$ Hence $\Gamma_{2^m} {(k,n)}\le \lambda 2^{2m} \Gamma_{2^{m+1}}^{1/2}{(k,n)}$. Set $x_m:= (\Gamma_{2^m}{(k,n)})^{2^{-m}}$. We get that $x_m \le \lambda^{2^{-m}} 2^{m 2^{-m+1}}\, x_{m+1}$, which implies that for any $m_0\ge 1$ and $n> m_0$, $$x_{m_0} \le \prod_{m=m_0}^n \lambda^{2^{-m}} 2^{m 2^{-m+1}}\, x_{n+1} \le c'' \, \lambda^{2^{-m_0+1}} ,$$ 

\noindent where in the last inequality, $c''$ denotes some positive numerical constant, and we use the fact that $\Gamma_i{(k,n)}\le  (n+1)^2$  for any $i$; hence $x_{n+1} \le (n+1)^{2^{-n}}$ is bounded in $n$.  

We have thus shown that for any $m_0\ge 1$, for all $n\ge 1$,  $\Gamma_{2^{m_0}}{(k,n)} \le (c'')^{2^{m_0}}\, \lambda^2= (c'')^{2^{m_0}}\,  (c')^2\, k^{4-\gamma}.$ Letting $n\to \infty$, we see that \eqref{eq:xi-alpha1} holds for $m=2^{m_0}$. The general case of  \eqref{eq:xi-alpha1}  for all $m$ then follows from  the monotonicity in $m$.

For \eqref{eq:xi-alphaL2}, since $\sum_{i=0}^\infty(\xi_i^{\alpha_i}-{\xi_i^\infty})= \sum_{k=0}^\infty \sum_{i=0}^\infty(\xi_i^{k}-{\xi_i^\infty}) \1_{\{\d(u_0,u_i)\in \{{2k, 2k+1}\}\}}$, we get that $$\|\sum_{i=0}^\infty(\xi_i^{\alpha_i}-{\xi_i^\infty})\|_{2} 
\le
\sum_{k=0}^\infty \| \sum_{i=0}^\infty(\xi_i^{k}-{\xi_i^\infty}) \1_{\{\d(u_0,u_i)\in \{{2k, 2k+1}\}\}}\|_2.$$

Let $d>16$, and fix an arbitrary $\gamma \in (6, \frac{d-4}2)$. For $k\ge1$, applying \eqref{eq:xi-alpha1} to $m=3$ gives that $\| \sum_{i=0}^\infty(\xi_i^{k}-{\xi_i^\infty}) \1_{\{\d(u_0,u_i)\in \{{2k, 2k+1}\}\}}\|_2 \lesssim k^{2- \gamma/2}$. Since $\gamma>6$, the sum over $k$ converges and we get  \eqref{eq:xi-alphaL2}. 

Finally,  \eqref{eq:xi-alphaL2-bis} follows from \eqref{eq:xi-alphaL2}, Corollary \ref{coro:var} and the fact that $ \sum_{i=1}^n \braces{\xi_i^\infty} =  \braces{Y_n}$. \end{proof}

\begin{lemma}\label{lem:I1}
Let $d>16$, then
\[
|\xifour_1| \lesssim n+\sqrt{\Theta_n},
\] where $\xifour_1$ and $\Theta_n$ are defined in \eqref{def-I1} and \eqref{def-Theta_n}, respectively.
\end{lemma}

\begin{proof}
Note that $$\xifour_1 =\sum_{\a=1}^{n} \e \Big[\braces{\xi_0^\infty} \braces{\xi_\a^\infty - \xi_\a^{\alpha_\a} }\big(\sum_{\b=1}^n \braces{\xi_\b^\infty} \big)^2\Big]
= \sum_{\a=1}^{n} \e \Big[\braces{\xi_0^\infty} \braces{\xi_\a^\infty - \xi_\a^{\alpha_\a} }\braces{Y_n}^2\Big].$$

Since $|\langle \xi_0^\infty\rangle|\le 1$, we have 
\begin{align}    
 |\xifour_1|   
	&\le
	\sum_{\a=1}^{n} \e \Big[  |\braces{\xi_\a^\infty - \xi_\a^{\alpha_\a} }|\, \braces{Y_n}^2 \Big]  \nonumber
	\\
	&\le
	 \sum_{\a=1}^{n} \e  [\xi_\a^{\alpha_\a}-{\xi_\a^\infty  }] \e\Big[\braces{Y_n}^2\Big]
	 +\sum_{\a=1}^{n} \e \Big[  (\xi_\a^{\alpha_\a}-\xi_\a^\infty) \, \braces{Y_n}^2\Big]
	 \nonumber
	 \\
	 & =:\xifour_1^{(1)}+\xifour_1^{(2)}.\label{eq:I1.1}
\end{align}

For $d>8$, we have $\sum_{\a=1}^{\infty} \e  [\xi_\a^{\alpha_\a}-{\xi_\a^\infty  }]<\infty $ by \eqref{compa-alpha}, and  $\e\left[\braces{Y_n}^2\right]\lesssim n$ by Corollary \ref{coro:var}, thus 
$\xifour_1^{(1)}\lesssim n$. 

For $\xifour_1^{(2)}$, we apply the Cauchy-Schwarz inequality to obtain that $$
\xifour_1^{(2)} \le \sqrt{\Theta_n} \sqrt{ \e \Big[\Big(\sum_{\a=1}^n  (\xi_\a^{\alpha_\a}-\xi_\a^\infty)\Big)^2\Big]}
\lesssim \sqrt{\Theta_n},
$$

\noindent by using \eqref{eq:xi-alphaL2} in the last estimate.  Lemma \ref{lem:I1} follows.
\end{proof}

\begin{lemma}\label{lem:I2}
When $d>16$,
\[
|\xifour_2|\lesssim n+\sqrt{\Theta_n},
\] where   $\xifour_2$ and $\Theta_n$ are defined in \eqref{def-I2} and \eqref{def-Theta_n}, respectively.

\end{lemma}
\begin{proof}
Since $|\braces{\xi_0^\infty} |\le1$ and $\braces{Y_n}= \sum_{\c=1}^{n} \braces{\xi_\c^\infty}$, we get that \begin{align} |\xifour_2|
&\le
 \e \Big[ \big|\sum_{\a=1}^n \braces{\xi_\a^{\alpha_\a} }\big| \times
 \big|\sum_{\b=1}^n \braces{\xi_\b^\infty-\xi_\b^{\alpha_\b}}\big|
 \times
 \big|\braces{Y_n}\big|\Big]  \nonumber
 \\
 &\le
 \e \Big[  \big|\braces{Y_n}\big| \times  \Big({\sum_{\b=1}^n \braces{\xi_\b^\infty-\xi_\b^{\alpha_\b}}}\Big)^2  \Big]+
   \e \Big[  \braces{Y_n}^2 \times \Big|{\sum_{\b=1}^n \braces{\xi_\b^\infty-\xi_\b^{\alpha_\b}}}\Big|   \Big] 
 \nonumber
 \\
 &=: \xifour_2^{(1)} + \xifour_2^{(2)},\label{II-1and2}\end{align}
where to obtain the second inequality, we use $\braces{\xi_\a^{\alpha_\a} }=   \braces{\xi_\a^{\alpha_\a}  - \xi_\a^\infty}  +  \braces{ \xi_\a^\infty} $. 

Note that $\xifour_2^{(2)} \le \xifour_1^{(1)}+\xifour_1^{(2)}$ (see \eqref{eq:I1.1}),  hence $\xifour_2^{(2)}\lesssim n+\sqrt{\Theta_n}$;  it suffices to handle $\xifour_2^{(1)}$.
By the trivial bounds: $| \braces{Y_n} |\le n$ and $\e[\braces{X}^2]\le\e[X^2]$,
$$\xifour_2^{(1)}
\le  n\e \left[\pars*{\sum_{\b=1}^n {(\xi_\b^{\alpha_\b}-\xi_\b^\infty)}}^2 \right]
\lesssim n,$$ by applying \eqref{eq:xi-alphaL2}. This completes the proof of Lemma \ref{lem:I2}. 
\end{proof}

\begin{lemma}\label{lem:I3}
When $d>16$,
\[
|\xifour_3|\lesssim n+\sqrt{\Theta_n},
\] where   $\xifour_3$ and $\Theta_n$ are defined in \eqref{def-I3} and \eqref{def-Theta_n}, respectively.
\end{lemma} 
\begin{proof} 
We bound $|\langle \xi_0^\infty\rangle|$ by $1$, use  $\sum_{\a=1}^{n} \braces{\xi_\a^{\alpha_\a} } = \sum_{\a=1}^{n} \braces{\xi_\a^{\alpha_\a} -\xi_\a^\infty}  + \braces{Y_n}$, and get that   \begin{align}     |\xifour_3|
& \le
 \e \Big[ \big( \sum_{\a=1}^{n} \braces{\xi_\a^{\alpha_\a} } \big)^2\,  | \sum_{\c=1}^{n} \braces{\xi_\c^\infty-\xi_\c^{\alpha_\c}}| \Big]
 \nonumber
 \\
 &\le
 2 \e \Big[    \Big| \sum_{c={1}}^{n} \braces{\xi_\c^\infty-\xi_\c^{\alpha_\c}}\Big|^3 \Big]  + 2
 \e \Big[ \braces{Y_n}^2\, \Big| \sum_{\c=1}^{n} \braces{\xi_\c^\infty-\xi_\c^{\alpha_\c}}\Big| \Big] \nonumber
 \\
 &\le
  2 n\e \Big[    \Big| \sum_{\c=1}^{n} \braces{\xi_\c^\infty-\xi_\c^{\alpha_\c}}\Big|^2 \Big]  + 2
  \e \Big[   \braces{Y_n}^2\, \Big| \sum_{\c=1}^{n} \braces{\xi_\c^\infty-\xi_\c^{\alpha_\c}}\Big| \Big] \nonumber
  \\
 &=:\xifour_3^{(1)}+ \xifour_3^{(2)}. \end{align}

By applying  \eqref{eq:xi-alphaL2}, $\xifour_3^{(1)} \lesssim n$, while  $\xifour_3^{(2)}=2\xifour_2^{(2)}$ by \eqref{II-1and2}, which has already been shown to satisfy  $\xifour_3^{(2)} \lesssim n+\sqrt{\Theta_n}$. 
\end{proof}

\begin{lemma}\label{lem:I4}
When $d>16$,
\[
 \xifour_4 \lesssim n+\sqrt{\Theta_n},
\] where   $\xifour_4$ and $\Theta_n$ are defined in \eqref{def-I4} and \eqref{def-Theta_n}, respectively.
\end{lemma}

\begin{proof}
By comparing $k_1,k_2,k_3$, we can subdivide $\xifour_4$ into
\[
\xifour_4=3\xifour_4^{(1)}+3\xifour_4^{(2)}+\xifour_4^{(3)},
\]
where
\begin{align*} 
\xifour_4^{(1)}&:= \sum_{\a,\b,\c=1}^{n} \sum_{k_1,k_2,k_3=0}^\infty
\e \left[\braces{\xi_0^\infty-\xi_0^{k_1}} \braces{\xi_\a^{k_1}\1_{\{\alpha_\a=k_1\}}} \braces{\xi_\b^{k_2}\1_{\{\alpha_\b=k_2\}}}  \braces{\xi_\c^{k_3}\1_{\{\alpha_\c=k_3\}}}\1_{\{k_1<k_2\wedge k_3\}} \right]  , \\
\xifour_4^{(2)}&:= \sum_{\a,\b,\c=1}^{n} \sum_{k_1,k_2,k_3=0}^\infty
\e \left[\braces{\xi_0^\infty-\xi_0^{k_1}} \braces{\xi_\a^{k_1}\1_{\{\alpha_\a=k_1\}}} \braces{\xi_\b^{k_2}\1_{\{\alpha_\b=k_2\}}}  \braces{\xi_\c^{k_3}\1_{\{\alpha_\c=k_3\}}}\1_{\{k_1=k_2<k_3\}} \right]  , \\
\xifour_4^{(3)}&:= \sum_{\a,\b,\c=1}^{n} \sum_{k_1,k_2,k_3=0}^\infty
\e \left[\braces{\xi_0^\infty-\xi_0^{k_1}} \braces{\xi_\a^{k_1}\1_{\{\alpha_\a=k_1\}}} \braces{\xi_\b^{k_2}\1_{\{\alpha_\b=k_2\}}}  \braces{\xi_\c^{k_3}\1_{\{\alpha_\c=k_3\}}}\1_{\{k_1=k_2=k_3\}} \right]  .
\end{align*}

We first estimate $\xifour_4^{(1)}$. Note that 
$$\xifour_4^{(1)}= 
\sum_{\a=1}^{n} \sum_{k=0}^\infty
\e \Big[\braces{\xi_0^\infty-\xi_0^{k}} \braces{\xi_\a^{k}\1_{\{\alpha_\a=k\}}} \big(\sum_{\b=1}^n\braces{\xi_\b^{\alpha_\b}\1_{\{\alpha_\b>k\}}}\big)^2 \Big]. 
$$

For nonnegative random variables $X,Y,Z\ge 0$,
\begin{align*}
-\e\bracks{\braces{X}\braces{Y}Z}
\le \e[X] \e[Y Z] + \e[Y] \e[X Z] 
\end{align*}
Apply this to
\[
X=\xi_0^{k}-\xi_0^\infty, \,
Y=\xi_\a^{k}\1_{\{\alpha_\a=k\}},\,
Z=\big(\sum_{\b=1}^n\braces{\xi_\b^{\alpha_\b}\1_{\{\alpha_\b>k\}}}\big)^2,
\]

\noindent we have \begin{align} \xifour_4^{(1)}
&
\le 
\sum_{\a=1}^{n} \sum_{k=0}^\infty
\e[\xi_0^k-\xi_0^\infty] \, \e \Big[  \xi_\a^{k}\1_{\{\alpha_\a=k\}} \big(\sum_{\b=1}^n\braces{\xi_\b^{\alpha_\b}\1_{\{\alpha_\b>k\}}}\big)^2 \Big]  \nonumber
\\ & \qquad + \sum_{\a=1}^{n} \sum_{k=0}^\infty
\e[\xi_\a^{k}\1_{\{\alpha_\a=k\}} ] \, \e \Big[ (\xi_0^k-\xi_0^\infty)  \big(\sum_{\b=1}^n\braces{\xi_\b^{\alpha_\b}\1_{\{\alpha_\b>k\}}}\big)^2 \Big],  \nonumber
\\
&=: A_{\eqref{I41-2sums}} + B_{\eqref{I41-2sums}}. \label{I41-2sums}\end{align}

Let us estimate  $A_{\eqref{I41-2sums}} $.  By {Lemma \ref{lem:xi_to_infty}}, $\e[\xi_0^k-\xi_0^\infty] \lesssim (1+k)^{2-d/2}$ for all $k\ge 0$. Observe that $0\le \xi_\a^k\le 1$, and that
\begin{align*}
\braces{\xi_\b^{\alpha_\b}\1_{\{\alpha_\b>k\}}}
=
\braces{\xi_\b^\infty}+\braces{(\xi_\b^{\alpha_\b}-\xi_\b^\infty)\1_{\{\alpha_\b>k\}}}-\braces{\xi_\b^\infty\1_{\{\alpha_\b\le k\}}}.
\end{align*}

Since $\sum_{\b=1}^n\braces{\xi_\b^{\infty}}=\braces{Y_n}$, we have \begin{equation} \big(\sum_{\b=1}^n\braces{\xi_\b^{\alpha_\b}\1_{\{\alpha_\b>k\}}}\big)^2  
\le
 3 \braces{Y_n}^2+ 3 \big(\sum_{\b=1}^n\braces{(\xi_\b^{\alpha_\b}-\xi_\b^\infty)\1_{\{\alpha_\b>k\}}}\big)^2
 + 3 \big(\sum_{\b=1}^n \braces{\xi_\b^\infty\1_{\{\alpha_\b\le k\}}}\big)^2.  \label{sum-b-3terms}\end{equation}

We claim that for $d>16$, 
\begin{align} 
\sum_{\a=1}^{n} \sum_{k=0}^\infty
(1+k)^{2-d/2} \, \e \Big[ \1_{\{\alpha_\a=k\}} \braces{Y_n}^2 \Big]
& \lesssim \sqrt{\Theta_n}, \label{I41-A1}
\end{align}
\begin{align} 
\sum_{\a=1}^{n} \sum_{k=0}^\infty
(1+k)^{2-d/2} \, \e \Big[ \1_{\{\alpha_\a=k\}} \big(\sum_{\b=1}^n\braces{(\xi_\b^{\alpha_\b}-\xi_\b^\infty)\1_{\{\alpha_\b>k\}}}\big)^2
\Big]
& \lesssim n, \label{I41-A2}
\end{align}
\begin{align} 
 \sum_{\a=1}^{n} \sum_{k=0}^\infty
(1+k)^{2-d/2} \, \e \Big[ \1_{\{\alpha_\a=k\}} \big(\sum_{\b=1}^n \braces{\xi_\b^\infty\1_{\{\alpha_\b\le k\}}}\big)^2
\Big]
& \lesssim 1 .  \label{I41-A3} \end{align}

\noindent Then we will obtain that for $d>16$, \begin{equation} A_{\eqref{I41-2sums}} \lesssim  n+ \sqrt{\Theta_n} .    \end{equation}

We first prove \eqref{I41-A1}. 
By the Cauchy-Schwarz inequality, the left-hand-side of \eqref{I41-A1} is less than \begin{align*}      \sum_{k=0}^\infty
(1+k)^{2-d/2} \, \Big(\e\Big[\big(\sum_{\a=1}^n \1_{\{\alpha_\a= k\}} \big)^2\Big]\Big)^{1/2}\, \Theta_n^{1/2} 
 \lesssim  \Theta_n^{1/2} \, \sum_{k=0}^\infty
(1+k)^{3-d/2} ,  
\end{align*}

\noindent where the last inequality follows from \eqref{eq:moment_geo2}. Since $d> 8$, the above sum over $k$ converges and we get \eqref{I41-A1}.

By expanding $\braces{\cdot}$ and dropping $\1_{\{\alpha_\b>k\}}\le 1$, we see that \eqref{I41-A2} is bounded, up to a multiplicative factor of $2$, by
\begin{align*}  & \sum_{\a=1}^{n} \sum_{k=0}^\infty
(1+k)^{2-d/2} \, 
\Big(\e \Big[ \1_{\{\alpha_\a=k\}} \big(\sum_{\b=1}^n (\xi_\b^{\alpha_\b}-\xi_\b^\infty) \big)^2\Big]    
+ \p(\alpha_\a=k) \Big(\e\Big[ \sum_{\b=1}^n (\xi_\b^{\alpha_\b}-\xi_\b^\infty)\Big]  \Big)^2 \Big)
\\ 
\le & n \, \e \Big[   \big(\sum_{\b=1}^\infty (\xi_\b^{\alpha_\b}-\xi_\b^\infty)^2 \Big] + n\, \Big(\e\Big[ \sum_{\b=1}^\infty (\xi_\b^{\alpha_\b}-\xi_\b^\infty)\Big]  \Big)^2
 \lesssim n,
\end{align*}
where in the first inequality, we simply use $(1+k)^{2-d/2} \le 1$,  and the second inequality follows from \eqref{eq:xi-alphaL2}.

Similarly, \eqref{I41-A3} is, up to a multiplicative factor of $2$, less than \begin{align*}  &  \sum_{\a=1}^{n} \sum_{k=0}^\infty
(1+k)^{2-d/2} \, \Big(\e \Big[ \1_{\{\alpha_\a=k\}} \big(\sum_{\b=1}^n  \1_{\{\alpha_\b\le k\}}\big)^2
\Big]
  +  \p(\alpha_\a=k) \Big(\e\Big[\sum_{\b=1}^n  \1_{\{\alpha_\b\le k\}}\Big]\Big)^2\Big).\end{align*}

By  \eqref{eq:moment_geo1} and \eqref{eq:moment_geo2}, for any $p\ge1$,
\begin{align} \label{eq:moment_geo1-alpha}
\e \Big[\Big( \sum_{\b=1}^\infty \1_{\{\alpha_\b \le k \}} \Big)^p\Big] & \lesssim_p  (1+k)^{2 p}, \quad
\e \Big[\Big( \sum_{\b=1}^\infty \1_{\{\alpha_\b = k \}} \Big)^p \Big] \lesssim_p (1+k)^{p}.\end{align}

Therefore, we use the Cauchy-Schwarz inequality and obtain that  \begin{align*}     & \sum_{\a=1}^{n} \sum_{k=0}^\infty
(1+k)^{2-d/2} \, \e \Big[ \1_{\{\alpha_\a=k\}} \big(\sum_{\b=1}^n  \1_{\{\alpha_\b\le k\}}\big)^2
\Big]
\\ \le &\sum_{k=0}^\infty (1+k)^{2-d/2} \, \sqrt{\e \Big[\Big(\sum_{\a=1}^\infty \1_{\{\alpha_\a=k\}}\Big)^2\Big] \e \Big[\Big(\sum_{\b=1}^\infty  \1_{\{\alpha_\b\le k\}}\Big)^4\Big]} 
\lesssim 
\sum_{k=0}^\infty (1+k)^{2-d/2} \,  k^{1+4}  ,
\end{align*}

\noindent whose sum over $k$ is finite for $d>16$. This yields \eqref{I41-A3}.

The term $B_{\eqref{I41-2sums}}$ can be estimated in a similar way; we only point out the main differences. Using \eqref{sum-b-3terms} and bounding $\xi_\a^k $ by $1$, we need to handle the  following three double sums: \begin{align*}   & \sum_{\a=1}^n \sum_{k=0}^\infty \p(\alpha_\a=k) \e [(\xi_0^k-\xi_0^\infty) \braces{Y_n}^2],\\
&
\sum_{\a=1}^n \sum_{k=0}^\infty  \p(\alpha_\a=k) \e\Big[(\xi_0^k-\xi_0^\infty)  \big(\sum_{\b=1}^n\braces{(\xi_\b^{\alpha_\b}-\xi_\b^\infty)\1_{\{\alpha_\b>k\}}}\big)^2\Big],
\\ & \sum_{\a=1}^n \sum_{k=0}^\infty  \p(\alpha_\a=k) \e\Big[(\xi_0^k-\xi_0^\infty)  \big(\sum_{\b=1}^n \braces{\xi_\b^\infty\1_{\{\alpha_\b\le k\}}}\big)^2\Big].
 \end{align*}

For the first term with $\braces{Y_n}^2$, we use the Cauchy-Schwarz inequality, {Lemma \ref{lem:xi_to_infty}} and 
\eqref{eq:moment_geo1-alpha} to see that the corresponding double sum is $O(\sqrt{\Theta_n})$. 

For the middle term with $(\xi_\b^{\alpha_\b}-\xi_\b^\infty)\1_{\{\alpha_\b>k\}}$, we drop $\xi_0^k-\xi_0^\infty \le 1$, $\1_{\{\alpha_\b>k\}}\le1$ and use \eqref{eq:xi-alphaL2} to see that the corresponding double sum is $O(n)$. 

Finally for the last term with $\xi_\b^\infty \1_{\{\alpha_\b\le k\}}$, we drop $\xi_\b^\infty\le 1$ and use the independence between $(\xi_0^k-\xi_0^\infty) $ and $\sum_{\b=1}^n  1_{\{\alpha_\b\le k\}}$, then we use \eqref{eq:moment_geo1-alpha}  to see that the corresponding double sum is $O(1)$. This implies that $ B_{\eqref{I41-2sums}} \lesssim  n+ \sqrt{\Theta_n} $, hence \begin{equation}   I_4^{(1)} \lesssim  n+ \sqrt{\Theta_n} .  \end{equation}

It remains to handle $\xifour_4^{(2)}$ and $\xifour_4^{(3)}$,  which are much easier. Note that
\begin{align*}
\xifour_4^{(2)}
=&\sum_{\a,\b,\c=1}^{n} \sum_{k_1,k_2,k_3=0}^\infty
\e \left[\braces{\xi_0^\infty-\xi_0^{k_1}} \braces{\xi_\a^{k_1}\1_{\{\alpha_\a=k_1\}}} \braces{\xi_\b^{k_2}\1_{\{\alpha_\b=k_2\}}}  \braces{\xi_\c^{k_3}\1_{\{\alpha_\c=k_3\}}}\1_{\{k_1=k_2<k_3\}} \right]\\
=& \sum_{k=0}^\infty
\e \left[\braces{\xi_0^\infty-\xi_0^{k}} \big( \sum_{\a=1}^{n}\braces{\xi_\a^{k}\1_{\{\alpha_\a=k\}}} \big)^2  \big(\sum_{\c=1}^{n} \braces{\xi_\c^{\alpha_\c}\1_{\{\alpha_\c>k\}}}\big) \right],
\end{align*}
and
\begin{align*}
\xifour_4^{(3)}&=\sum_{\a,\b,\c=1}^n\sum_{k=0}^\infty\e[\braces{\xi_0^\infty-\xi_0^k}\braces{\xi_\a^{k}\1_{\{\alpha_\a=k\}}}\braces{\xi_\b^{k}\1_{\{\alpha_\b=k\}}}\braces{\xi_\c^{k}\1_{\{\alpha_\c=k\}}}]
\\
&= \sum_{k=0}^\infty\e\Big[\braces{\xi_0^\infty-\xi_0^k} \big(\sum_{\a=1}^n\braces{\xi_\a^{k}\1_{\{\alpha_\a=k\}}}\big)^3\Big].
\end{align*}

For $\xifour_4^{(2)}$, we bound $|\sum_{\c=1}^{n} \braces{\xi_\c^{\alpha_\c}\1_{\{\alpha_\c>k\}}}|\le n$, drop $\xi_\a^k\le 1$, expand $\braces{\cdot}$, and use the independence between  $\xi_0^\infty-\xi_0^{k}$ and $\sum_{\a=1}^{n} \1_{\{\alpha_\a=k\}}$, to obtain that  \begin{align*} \xifour_4^{(2)}
&\lesssim   n\, \sum_{k=0}^\infty  \e [\xi_0^k-\xi_0^\infty] \, \e \Big[\big(\sum_{\a=1}^{n} \1_{\{\alpha_\a=k\}}\big)^2\Big]
\\
&\lesssim n \,\sum_{k=0}^\infty   (1+k)^{2-d/2} \, (1+k)^2 
\lesssim n.  \end{align*}

We estimate $\xifour_4^{(3)}$ in the same way, by dropping $\xi_\a^k\le 1$, applying \eqref{eq:moment_geo2} to $p=3$: \begin{align*} \xifour_4^{(3)}
&\lesssim     \sum_{k=0}^\infty  \e [\xi_0^k-\xi_0^\infty] \, \e \Big[\big(\sum_{\a=1}^{n} \1_{\{\alpha_\a=k\}}\big)\Big]^3
\\
&\lesssim   \,\sum_{k=0}^\infty   (1+k)^{2-d/2} \,(1+k)^3 
\lesssim 1.  \end{align*}

These two estimates complete the proof of Lemma \ref{lem:I4}. \end{proof}

Now by combining Lemmas \ref{lem:I0}, \ref{lem:I1}, \ref{lem:I2}, \ref{lem:I3} and \ref{lem:I4}, we immediately obtain Lemma \ref{lem:main4th}. Thus the proof of Proposition \ref{P:4thmoment} is complete.


\begin{thebibliography}{99}
\baselineskip=12pt

\bibitem{AsselahSchapiraSousi2018}
A. Asselah, B. Schapira, and A. Sousi. (2018). \newblock Capacity of the range of random walk on $\z^d$. \newblock \emph{Trans. Amer. Math. Soc.} \textbf{370}, 7627--7645 




\bibitem{AsselahSchapiraSousi2019}
A. Asselah, B. Schapira, and A. Sousi. (2019).
\newblock Capacity of the range of random walk on $\z^4$. 
\newblock \emph{Annals of Probability} \textbf{47}, 1447--1497.


\bibitem{athreya2012branching}
K.B. Athreya and P.E. Ney. (2012).
\newblock {\em Branching processes}, volume 196.
\newblock Springer Science \& Business Media.


\bibitem{bai-wan} 
T. Bai and  Y. Wan. (2022).
\newblock Capacity of the range of tree-indexed random walk.
\newblock \emph{Ann. Appl. Probab. } {\bf 32} 1557--1589.

\bibitem{benjamini-curien} 
I. Benjamini and N. Curien. (2012). 
\newblock{Recurrence of the $\z^d$-valued infinite snake via unimodularity. }
\newblock \emph{Electron. Commun. Probab. } {\bf 17} 1--10.

 
\bibitem{CyganSandricSebek2019}
W. Cygan, N. Sandrić, and S. Šebek. (2021).
\newblock Central limit theorem for the capacity of the range of stable random walks.
\newblock \emph{Stochastics}, \textbf{94}(2), 226--247.



\bibitem{dedecker02}
J. Dedecker and F. Merlev{\`e}de. (2002).
\newblock Necessary and sufficient conditions for the conditional central limit
  theorem.
\newblock {\em The Annals of Probability}, 30(3):1044--1081.

\bibitem{jain1971}
N.~C. Jain and W.~E. Pruitt. (1971). 
\newblock{The range of transient random walk}.
\newblock{\em Journal d'Analyse Math\'ematique} {\bf 24}, 369--393.

\bibitem{LeGall05}
J.F. Le Gall. (2005).
\newblock {Random trees and applications}.
\newblock {\em Probability Surveys}, {\bf 2}, 245 -- 311.

\bibitem{LeGall10}
J.F. Le Gall (2010).
\newblock {Itô’s excursion theory and random trees}.
\newblock {\em Stoch. Proc. Appl.}, {\bf 120}, 721 -- 749.

\bibitem{LeGall-Lin-lowdim}
J.F. Le Gall and S. Lin.  (2015).
\newblock The range of tree-indexed random walk in low dimensions.
\newblock {\em The Annals of Probability}.  {\bf 43}, 2701-- 2728.

\bibitem{LeGall-Lin-range}
J.F. Le~Gall and S. Lin. (2016).
\newblock The range of tree-indexed random walk.
\newblock {\em Journal of the Institute of Mathematics of Jussieu},
  15(2):271--317.
  
  \bibitem{LeGallRosen91}
J.F. Le Gall and J. Rosen. (1991).
\newblock The range of stable random walks.
\newblock \emph{Annals of Probability} \textbf{19}(2), 650--705.

\bibitem{marckert-mokkadem}
	J.F. Marckert and  A. Mokkadem.  (2003). 
\newblock	The depth first processes of Galton-Watson trees converge to the same Brownian excursion. 
\newblock {\it  Ann.  Probab.} {\bf 31} 1655--1678.
	

\bibitem{Rosenthal} H.P. Rosenthal. (1970).   On the subspaces of $L^p (p>2)$   spanned by sequences
of independent random variables. {\it Israel J. Math.}, {\bf 8},   273--303.


\bibitem{Schapira2020}
B. Schapira. (2020).
\newblock Capacity of the range in dimension five.
\newblock \emph{Annals of Probability} \textbf{48}(6), 2909--2947.


\bibitem{zhu2017critical}
Q. Zhu. (2017)
\newblock {Critical branching random walks, branching capacity and
  branching interlacements}.
\newblock PhD thesis, University of British Columbia.

\bibitem{Zhu-cbrw}
Q. Zhu. (2021).
\newblock {On the critical branching random walk III: The critical dimension}.
\newblock {\em Annales de l'Institut Henri Poincaré, Probabilités et
  Statistiques}  {\bf 57}, 73 -- 93.


\end{thebibliography}
\end{document}